\documentclass[10pt,leqno]{article}
\usepackage{dynkin-diagrams}
\usepackage{authblk}
\usepackage{amssymb}
\usepackage{amsrefs}
\usepackage{verbatim}
\usepackage{array}

\usepackage{tikz}

\usepackage{rotating}
\usepackage{amsmath}
\usepackage{tabularx}
\usepackage{caption}
\usepackage{theorem}
\usepackage[matrix,tips,frame,color,line,poly]{xy}

\newtheorem{theorem}[equation]{Theorem}
\newtheorem{corollary}[equation]{Corollary}
\newtheorem{definition}[equation]{Definition}
\newtheorem{lemma}[equation]{Lemma}

\newtheorem{proposition}[equation]{Proposition}
\newtheorem{remark}[equation]{Remark}
{\theorembodyfont{\rmfamily}

\newtheorem{remarkplain}[equation]{Remark}

\newtheorem{exampleplain}[equation]{Example}

}
\newcommand{\qed}{\hfill $\square$ \medskip}
\newenvironment{proof}[1][Proof]{\noindent\textbf{#1.} }{\qed}

\newcommand{\Aut}{\text{Aut}}

\renewcommand{\int}{\text{int}}
\newcommand{\h}{\mathcal H}
\newcommand{\D}{\mathcal D}
\renewcommand{\j}{\mathcal J}

\newcommand{\Hom}{\text{Hom}}

\newcommand{\Ad}{\text{Ad}}
\newcommand{\Lie}{\text{Lie}}

\newcommand{\caC}{\mathcal C}
\newcommand{\caA}{\mathcal A}

\newcommand{\R}{\mathbb R}
\newcommand{\C}{\mathbb C}
\newcommand{\Z}{\mathbb Z}

\newcommand{\N}{\text{N}}
\newcommand{\Q}{\mathbb Q}

\newcommand{\Gsc}{G_{\text{sc}}}

{\renewcommand{\h}{\mathfrak h}}

\renewcommand{\a}{\mathfrak a}

\newcommand{\ch}[1]{#1^\vee}

\renewcommand{\sec}[1]{\section{#1}
\renewcommand{\theequation}{\thesection.\arabic{equation}}
  \setcounter{equation}{0}}
\newcommand{\subsec}[1]{\subsection{#1}
\renewcommand{\theequation}{\thesubsection.\arabic{equation}}}
\renewcommand{\t}{\mathfrak t}

\newcommand{\g}{\mathfrak g}

\newcommand\inv{^{-1}}

\newcommand\wt{\widetilde}
\newcommand\wh{\widehat}

\newcommand{\Wconj}{[W]}
\newcommand{\Weconj}{[{W^e}]}
\newcommand{\Gssconj}{[{G^\text{ss}}]}
\renewcommand{\ss}{\text{ss}}
\newcommand{\e}{\text{e}}
\newcommand{\Gss}{G^\ss}
\newcommand{\Wext}{\negthinspace\negthinspace\phantom{a}^\delta W}

\newcommand{\Gext}{\negthinspace\negthinspace\phantom{a}^\delta G}

\newcommand{\Daffine}{\wt D(G,\delta)}

\def\fg{\mathfrak g}
\def\fc{\mathfrak c}

\def\ft{\mathfrak t}
\def\fs{\mathfrak s}

\def\ge{\geqslant}
\def\le{\leqslant}
\def\a{\alpha}
\def\b{\beta}
\def\G{\Gamma}
\def\d{\delta}

\def\p{\pi}

\def\t{\tau}
\def\th{\theta}

\def\l{\lambda}
\def\z{\zeta}

\def\x{\xi}
\def\i{^{-1}}
\newcommand{\twoAeven}{\phantom{}^2\negthinspace A_{2n}}
\newcommand{\twoAodd}{\phantom{}^2\negthinspace A_{2n+1}}
\newcommand{\twoD}{\phantom{}^2\negthinspace D_{n}}
\renewcommand{\sec}[1]{\section{#1}
\renewcommand{\theequation}{\thesection.\arabic{equation}}
  \setcounter{equation}{0}}
\font\temporary=manfnt
\def\dbend{{\temporary\char127}} 
\def\danger{\begin{trivlist}\item[]\noindent%
\begingroup\hangindent=3pc\hangafter=-2
\def\par{\endgraf\endgroup}%
\hbox to0pt{\hskip-\hangindent\dbend\hfill}\ignorespaces}
\def\enddanger{\par\end{trivlist}}

\begin{document}

\author{Jeffrey Adams}
\author{Xuhua He}
\author{Sian Nie}
\affil{Department of Mathematics, University of Maryland, jda@math.umd.edu}
\affil{Department of Mathematics, University of Maryland}
\affil{Current: The Institute of Mathematical Sciences and Department of Mathematics, The Chinese University of Hong Kong, Shatin, N.T., Hong Kong, xuhuahe@gmail.com}
\affil{Institute of Mathematics, Academy of Mathematics and Systems Science, Chinese Academy of Sciences, 100190, Beijing, China
  , niesian@amss.ac.cn}
\date{}                     

\title{From conjugacy classes in  the Weyl group to semisimple conjugacy classes}

\maketitle

\let\thefootnote\relax\footnotetext{X.~H.~was partially supported by
  NSF DMS-1801352. S.N. is supported in part by QYZDB-SSW-SYS007 and
  NSFC grant (Nos. 11501547, 11621061 and 11688101).}

\let\thefootnote\relax

\let\thefootnote\relax

\noindent{\bf Abstract:} Suppose $G$ is a connected complex semisimple group
and $W$ is its Weyl group. The lifting of an element of $W$ to $G$ is
semisimple. This induces a well-defined map from the set of elliptic
conjugacy classes of $W$ to the set of semisimple conjugacy classes of
$G$. In this paper, we give a uniform algorithm to compute this
map. We also consider the twisted case.

\begin{center}

  \large \emph{To Bert Kostant with admiration}
  \end{center}

\sec{Introduction}

Let $G$ be a connected complex semisimple group.
Choose a Cartan subgroup $T\subset G$, let $N(T)=N_G(T)$ be the normalizer of $T$ in $G$,
and let $W=N(T)/T$ be the Weyl group.
We have the exact sequence
\begin{equation}
\label{e:exactW}
1\rightarrow T\rightarrow N(T)\overset p\rightarrow W\rightarrow 1.
\end{equation}
In \cite{AH} we considered the question of whether this sequence splits,
and more generally if $w\in W$, what can be said about the orders of elements of $p\inv(w)$.
Here we consider a related problem.

Let $\Wconj$ be the set of conjugacy classes in $W$. Let $\Gss$ be the
semisimple elements of $G$, and $\Gssconj$ the set of semisimple
conjugacy classes. We have $N(T)\subset \Gss$. For $w\in W$ write $[w]$ for the $W$-conjugacy
class of $w\in W$. Similarly for $g\in \Gss$ let $[g]$ be the
$G$-conjugacy class of $g$.

Suppose $w\in W$ and $n_w\in p\inv(w)\in N(T)$. For general $w$ there
are many choices of $n_w$, even up to conjugacy.  However in an
important special case the choice of $n_w$ is unique up to
conjugation. We say $w\in W$ is {\it elliptic} if it has no nontrivial fixed
vectors in the reflection representation.  It is well known that if
$w$ is elliptic then any two elements of $p\inv(w)\subset N(T)$ are
$T$-conjugate (see Lemma \ref{l:conjugate}). 

\begin{definition}
\label{d:basic}
Suppose $w\in W$ is elliptic.
Choose $n_w\in p\inv(w)\subset N(T)$, and define
$$
\Psi: \Weconj\rightarrow\Gssconj: [w]\rightarrow [n_w].
$$
This map is well defined on the level of conjugacy classes by Lemma \ref{l:conjugate}.
\end{definition}

This map has been studied by many people, including Reeder
\cite{reeder_torsion} and Reeder-Levy-Yu-Gross \cite{rgly}, for
applications to representations of $p$-adic groups and number theory.
For the exceptional groups the map $\Psi$ has been computed using
case-by-case calculations: the papers \cite{levy_exceptional}, \cite{bouwknegt}
and \cite{rgly} cover all cases, with some overlaps, and include other
information. The main result of this paper is a uniform algorithm to compute
the map $\Psi$ which is free of any case-by-case considerations.

We give this algorithm in the next section, after first discussing the
twisted case. We have implemented the algorithm in the {\it Atlas of
  Lie groups and Representations} software \cite{atlas_software} (see
the file {\tt weyltosemisimple.at}). It can be used to compute $\Psi$
for any semisimple group, and Section \ref{s:tables} includes  complete tables for the
(twisted and untwisted) exceptional groups.

\subsec{The twisted case}
\label{s:twisted}

A pinning is a triple $(G,T,\{X_\alpha\mid \alpha\in\Pi\})$,
where $\Pi$ is a set of simple roots of $T$ in $G$, and for each
$\alpha\in\Pi$, $X_\alpha$ is an $\alpha$-weight vector.  
We say an automorphism of $G$ is {\it distinguished} if it preserves some pinning.  An
inner automorphism is distinguished if and only if it is trivial as
any two pinnings are conjugate by a unique inner automorphism,
and the
outer automorphism group of $G$ is isomorphic to the group of
automorphisms of a pinning.

Suppose $\Pi$ is a fixed set of simple roots, and
$(G,T,\{X_\alpha\mid\alpha\in\Pi\})$ is a pinning.
Suppose $\delta$ is an automorphism of $G$, preserving
the pinning.
Then $\delta$ induces an
automorphism of $W$, which we also denote by $\delta$. We define
$$
\Gext=G \rtimes\langle\delta\rangle, \quad \Wext=W\rtimes\langle\delta\rangle.
$$
Let $N_{\Gext}(T)$ be the normalizer of $T$ in $\Gext$. 
Then $\Wext\simeq N_{\Gext}(T)/T$, and we write $p:N_{\Gext}(T)\rightarrow \Wext$.

Let $\Delta=\Delta(T,G)$ be the root system, and
$V=\Q\langle\Delta\rangle$. This is a representation of $\Wext$ which
we refer to as the reflection representation.  We say $y\in\Wext$ is
elliptic if it has no nontrivial fixed vectors in the reflection
representation.  Write $[y]$ for the $\Wext$-conjugacy class of
$y$. We consider elements $y\in W\delta$, and write $[(W\delta)^e]$
for the $\Wext$-conjugacy classes of elliptic elements in $W\delta$.

From the identity $\delta\inv(\delta w)\delta=w(\delta w)w\inv$ it
follows that two elements of $W\delta$ are $\Wext$-conjugate if and
only if they are $W$-conjugate.

We say an element $y$ of $\Gext$ is {\it semisimple} if $\Ad(y)$ is
semisimple.  Write $[y]$ for the $\Gext$-conjugacy class of
$y\in \Gext$. We consider elements $y\in G\delta$, and write
$[(G\delta)^{\ss}]$ for the $G$-conjugacy classes of semisimple
elements in $G\delta$.  As in the case of $\Wext$, two elements of
$G\delta$ are $\Gext$-conjugate if and only if they are $G$-conjugate.

\begin{lemma}
\label{l:conjugate}
Suppose $w\in W\delta$ is elliptic.
Then any two
elements of $p\inv(w)\subset N_{\Gext}(T)$ are $T$-conjugate.
\end{lemma}

\begin{proof}
The result is known, e.g. see \cite[Remark 4.1.1]{debacker_reeder}. We include a proof for completeness. 

Suppose $g\in p\inv(w)\subset \N(T)\delta$. 
Then for $t\in T$,
$$
tgt\inv=tw(t\inv)g.
$$
Since $w$ is elliptic and $G$ is semisimple, the map $t\rightarrow tw(t\inv)$ has finite
kernel, so is surjective.

Now suppose $g_1,g_2\in p\inv(w)$.
Since $w\in W\delta$, $g_1,g_2\in G\delta$. Write $g_1=h_1\delta,g_2=h_2\delta$ with $h_1,h_2\in G$.
By the previous discussion choose $t$ so that $tw(t\inv)=h_2h_1\inv$. Then $tg_1t\inv=g_2$.
\end{proof}

\begin{definition}
 \label{d:basictwisted}
Suppose $w\in W\delta$ is elliptic.
Choose $n_w\in p\inv(w)\subset N(T)\delta$, and define
$$
\Psi: [(W\delta)^\e]\rightarrow[(G\delta)^{\ss}]: [w]\rightarrow [n_w].
$$
This map is well defined on the level of conjugacy classes by Lemma \ref{l:conjugate}.
\end{definition}

If $w\in W\delta$ is not elliptic then the lifts $n_w$ are not all $G$-conjugate. Nevertheless
by realizing $w$ as an elliptic element in a Levi subgroup one can define a canonical map
from $[(W\delta)^\e]$ to $[(G\delta)^{\ss}]$. See Section \ref{s:nonelliptic}.

We've stated the result over $\C$. The algorithm applies over any
algebraically closed field of characteristic $0$, and in general with
a weak restriction on the characteristic. See Remark
\ref{r:otherfield}.

The authors thank Tom Haines for numerous helpful discussions and thank the referees for their suggestions and comments. 

\sec{The Algorithm}
\label{const}

Assume $G$ is a semisimple reductive group defined over $\C$, and
$\delta$ is an automorphism of $G$ of finite order. We fix a pinning
$(G,T,\{X_\alpha\mid\alpha\in\pi\})$ that is preserved by $\d$.   We write $\Delta$ for
the root system and $V=\Q\langle\Delta\rangle$.  Write $\ch\alpha$ for
the coroot associated to $\alpha\in\Delta$, $\ch\Delta$ the set of
coroots, and $\ch V_\Q=\Q\langle\ch\Delta\rangle$.
Write $\langle\,,\,\rangle$ for the canonical pairing between $\ch V$ and $V$.

Suppose $w\in \Wext$. Let
\begin{subequations}
\renewcommand{\theequation}{\theparentequation)(\alph{equation}}
\label{e:algorithm}
\begin{equation}
\Gamma_w=\{0\le\theta\le\pi\mid\text{such that } e^{i\theta}\text{ is an eigenvalue of $w$\text{ on }}V_\C.\}
\end{equation}
Write
\begin{equation}
\label{e:thetas}
\Gamma_w=\{\theta_1,\theta_2,\dots,\theta_k\}\text{ with }0\le \theta_1<\theta_2<\dots<\theta_k\le\pi.
\end{equation}
For $\theta\in \Gamma_w$ let
\begin{equation}
V(w,\theta)=\{v\in V_\R\mid w(v)+w\inv(v)=2\cos(\theta)v\}.
\end{equation}
Then $V(w,\theta)_\C$ is the direct sum of the eigenspaces of
$w$ on $V_\C$ with eigenvalues $e^{\pm i\theta}$.

For $1 \le i \le k$ set
$$
F_i=\sum_{j=1}^i V(w,\theta_j)
$$
and set $F_0=0,\theta_0=0$. This gives a filtration
\begin{equation}
\label{e:F_k}
0=F_0\subsetneqq F_1 \subsetneqq \dots\subsetneqq F_{k}=V
\end{equation}
with strict containments.
For $0 \le i \le k$, set
\begin{equation}
\Delta_i=\{\alpha\in\Delta\mid \langle\ch\alpha,F_i\rangle=0\}.
\end{equation}
Then each $\Delta_i$ is a root system, and set
\begin{equation}
W_i=W(\Delta_i).
\end{equation}
Thus we have
\begin{equation}
  \label{e:levis}
  \begin{aligned}
\Delta&=\Delta_0\supseteq \Delta_1\supseteq \dots\supseteq \Delta_k=\emptyset,\\
W&=W_0\supseteq W_1\supseteq \dots\supseteq W_k=\{1\}.\\
\end{aligned}
\end{equation}

By \cite{he_nie_minimal_finite}*{Lemma 5.1} after conjugating by $W$
we may assume  that all the $\Delta_i$ are standard, i.e., the
corresponding Levi subgroups are standard Levi subgroups.

For each $i$ set
\begin{equation}
\begin{aligned}
\Delta^+_i&=\Delta^+\cap \Delta_i,\\
\ch\rho_i&=\frac12\sum_{\alpha\in\Delta_i^+}\ch\alpha.
\end{aligned}
\end{equation}
We define rational coweights
$\{\ch\lambda_0,\dots, \ch\lambda_{k}\}$ by downward induction.
Set $\ch\lambda_k=0$, and for $0\le j\le k-1$ define
\begin{equation}\label{def-1}
\begin{aligned}
\ch\lambda_{j}=\frac{d(\theta_{j+1}-\theta_{j})}{2\pi}\ch\rho_j+\overline{\ch\lambda_{j+1}},
\end{aligned}
\end{equation}
where $d$ is the order of $w$, and $\overline{\ch\lambda_{j+1}}$ is the element in the $W_j$-orbit of $\ch\lambda_{j+1}$ which is dominant for $\Delta^+_j$.
\end{subequations}

\begin{theorem}
\label{t:main}
Suppose $w\in W\delta$ is of order $d$. Construct $\ch\lambda_0$ by the algorithm. Then
some lift $n_w$ of $w$ is $G$-conjugate to $\exp(2\pi \sqrt{-1}\ch\lambda_0/d)\delta$.

In particular if $w\in (W\delta)^e$ then 
\begin{equation}
  \label{e:main}
  \Psi([w])=[\exp(2\pi \sqrt{-1}\ch\lambda_0/d)\delta].
\end{equation}
\end{theorem}

\begin{remarkplain}
\label{r:full}
The algorithm  uses the full sequence $(\th_1, \ldots, \th_k)$.
One may replace the full sequence by any subsequence that is
admissible in the sense of \cite{he_nie_minimal_finite}*{subsection
  5.2}.  The proof of the Theorem applies with no change in this
generality. The element $\exp(2\pi \sqrt{-1}\ch\lambda_0/d)\delta \in \Gext$ 
obtained from an admissible subsequence of $(\th_1, \ldots, \th_k)$
is, in general, different from the element obtained here using the
full sequence.
\end{remarkplain}

We will use the following well-known conjugacy result \cite{serre_linear}*{Chapter 13, Prop. 2.5}.
We include a proof since it plays a role in the subsequent Proposition.

\begin{proposition}
Suppose $W$ is a finite group all of whose characters take values in $\Q$. 
Suppose $w\in W$ has order $d$. If $(d,k)=1$, then 
 $w^k$ is conjugate to $w$.
\end{proposition}

It is well known that all representations of Weyl group take integral values \cite{springer_regular}*{Theorem 8.5}.

\medskip

\begin{proof}
Let $F=\Q(\zeta_d)$ where $\zeta_d$ is a primitive $d^{th}$ root of unity.
If $\pi$ is a representation of $W$, with character $\theta_\pi$, then
$$
\theta_\pi(w)=\sum_{i=1}^r z_i,
$$
where each $z_i$ is a $d^{th}$ root of unity in $F$. Then
$$
\theta_\pi(w^k)=\sum_{i=1}^r z_i^k.
$$
Since $(k,d)=1$ the map $\zeta_d\rightarrow \zeta_d^k$ induces an automorphism $\tau$ of
$F/\Q$. By assumption  $\sum_i z_i\in \Q$, so
$$
\theta_\pi(w^k)=\sum_i z_i^k=\sum_i \tau(z_i)=\tau(\sum z_i)=\sum z_i=\theta_\pi(w).
$$
Since the characters separate conjugacy classes this implies $w$ is conjugate to $w^k$.

\end{proof}

We need a slight generalization of this.

\begin{proposition}
  \label{p:g^k}
Suppose $W$ is as in the previous Proposition, and  $\delta$
is an automorphism of $W$ of finite order. 
Let $\Wext=W\rtimes\langle\delta\rangle$. Suppose
$w\in W\delta\subset \Wext$ has order $d$.
Suppose $(k,d)=1$ and $w^k$ is also contained in $W\delta$. 
Then $w^k$ is conjugate to $w$.
\end{proposition}

Recall (Section \ref{s:twisted}) the statements with $W$-conjugacy or
$\Wext$-conjugacy are equivalent.

\begin{proof}
  Let $m$ be the order of $\delta$. By Clifford theory the characters of $\Wext$ are defined over
  $\Q(\zeta_m)$ where $\zeta_m$ is a primitive $m^{th}$ root of $1$.
  Replace $\Q$ with $E=\Q(\zeta_m)$ in the proof of the previous
  Proposition, and consider the field $E(\zeta_d)$. 
  The condition $w^k\in W\delta$ implies $k=1\pmod m$, so $\zeta_m^k=\zeta_m$ and 
  the map $\zeta_d\rightarrow \zeta_d^k$ induces an automorphism of $E(\zeta_d)/E$.
The  rest of the proof goes through with minor changes.
\end{proof}

\begin{remarkplain}
\label{r:otherfield}
Suppose $F$ is an algebraically closed field and $G$ is a connected semisimple algebraic group over $F$. Assume furthermore that the order of $\Wext$ is invertible in $F$.
Then all of the elements of $N_{\Gext}(T)$ are semisimple
and a version of Theorem \ref{t:main} holds in this setting. The only issue is to make
sense of the right hand side of \eqref{e:main} over $F$. 

Suppose $\ch\lambda_0$ is constructed as in the algorithm. Choose $m\in\Z$ so that $m\ch\lambda_0\in X_*$,
and choose a primitive $dm^{th}$ root of unity $\zeta_{dm}\in F$. Then
$$
\Psi([w])=[\delta (m\ch\lambda_0)(\zeta_{dm})]
$$
where we view $m\ch\lambda_0$ as a one-parameter subgroup
corresponding to $m\ch\lambda_0:F^\times\rightarrow T(F)$. By
Proposition \ref{p:g^k} this is independent of the choices of $m$ and
$\zeta_{dm}$.

\end{remarkplain}

\sec{Digression on  good elements}
\label{s:digression}

The algorithm in Section \ref{const} is  motivated in part by
the construction of the good elements in a given conjugacy class of
$\Wext$.
The good elements in $W$ were introduced by Geck and Michel in
\cite{geck_michel_good} and the notion was generalized to $\Wext$ by Geck, Kim and
Pfeiffer in \cite{gkp}.
In this section we discuss this
construction.

Let $B^+$ be the braid monoid associated with $(W,S)$. There is a
canonical injection $j:W\longrightarrow B^+$ identifying the
generators of $W$ with the generators of $B^+$ and
satisfying $j(w_1w_2)=j(w_1)j(w_2)$ for $w_1,w_2\in W$ whenever
$\ell(w_1w_2)=\ell(w_1)+\ell(w_2)$.

Now the automorphism $\delta$ induces an automorphism of $B^+$, which
is still denoted $\delta$. Set ${}^{\delta} B^+=B^+ \rtimes
\langle\delta\rangle$. Then $j$ extends in a canonical way to an
injection $\Wext \to {}^{\delta} B^+$, which we still denote by
$j$. We will simply write $\underline w$ for $j(w)$.

By definition, $w \in \Wext$ is a {\it good element} if there exists a
strictly decreasing sequence $\Pi_0 \supsetneq \Pi_1 \supsetneq \cdots
\supsetneq \Pi_l$ of subsets of $\Pi$ and even positive integers
$d_0,\cdots,d_l$ such
that $${\underline{w}}^d=\underline{w_0}^{d_0}\cdots\underline{w_l}^{d_l}.$$
Here $d$ is the order of $w$ and $w_i$ is the longest element of the
parabolic subgroup of $W$ generated by $\Pi_i$.

It was proved in \cite{geck_michel_good}, \cite{gkp} and \cite{he_minimal_length_double_cosets} that for any
conjugacy class of $\underline W$, there exists a good minimal length
element. In \cite{he_nie_minimal_finite}, the second and third-named authors
gave a general proof, which also provides an explicit construction of good
minimal length elements.

Now we recall the construction in \cite{he_nie_minimal_finite}.
Let $F_0,\dots, F_k$ be as in \eqref{e:F_k}.

Let $\caA$ be a Weyl chamber and for $0 \le i < k$, let $\caC_i(\caA)$
be the connected component of $V-\cup_{H_\a; F_i \subset H_\a}H$
containing $\caA$. We say $\caA$ is in {\it good position} with
respect to $w$ if for any $i$, the closure $\overline{\caC_i(\caA)}$
contains some regular point of $F_{i+1}$.

By \cite{he_nie_minimal_finite}*{Lemma 5.1}, there exists some Weyl chamber $\caA$ that is
in good position with respect to any given $w$. By definition, for any
$x \in W$, the Weyl chamber $x(\caA)$ is in good position with respect
to $x w x^{-1}$. In particular, for any conjugacy class of $\Wext$,
there exists an element $w$ such that the dominant chamber is in good
position with respect to $w$. In this case, $\Delta_i$ is the root system
of the standard Levi subgroup of $G$ associated to the subset
$\Pi_i:=\Pi \cap \Delta_i$ of simple roots and $W_i=W_{F_i}$ is a standard
parabolic subgroup of $W$ for any $i$. We denote by $W^{\Pi_i}$ (resp. ${}^{\Pi_i} W$) the
set of minimal length representatives in $W/W_i$ (resp. in $W_i \backslash W$). We write ${}^{\Pi_1} W^{\Pi_2}$ for ${}^{\Pi_1} W \cap W^{\Pi_2}$. By \cite{he_nie_minimal_finite}*{Proposition
2.2}, we have the following good factorization of $w$.

\begin{proposition}\label{good-fac} Let $w \in W \d \subset
\Wext$. Suppose that the dominant chamber is in good position with
respect to $w$. Then there are $x_i\in W$ $(1\le i\le k)$ so that 
$$
w=\d x_1 x_2 \cdots x_k,
$$ where for $1 \le i
\le k$, we have $x_i \in W_{i-1} \cap W^{\Pi_i}$.

Furthermore
$$
(\d x_1\cdots x_i)(\Pi_i)=\Pi_i\quad (1\le i\le k).
$$

\end{proposition}

The following result is proved in \cite{he_nie_minimal_finite}*{Theorem 5.3}.

\begin{theorem}
\label{t:braid}
Suppose $w \in \Wext$, and the fundamental chamber is
in good position with $w$. Then we have the following equality in the
Braid monoid associated with $(W, S)$: $$\underline w^d=\underline
w_0^{d \theta_1/\pi} \underline w_1^{d (\theta_2-\theta_1)/\pi} \cdots
\underline w_{k-1}^{d (\theta_{k}-\theta_{k-1})/\pi},$$ where $d$ is
the order of $w$ in $\Wext$, $(\theta_1, \cdots,
\theta_k)$ is the sequence consisting of the elements in $\G_{w}$
and $w_i$ is the maximal element in the standard parabolic subgroup
$W_i$.
\end{theorem}

\sec{The regular case}
\label{s:regular}

We first study the regular elliptic elements. 
A similar discussion is in \cite{reeder_torsion}*{Section 2.6}.

Following
\cite{springer_regular} we say an element $w \in {}^{\delta} W$ is {\it regular} if it has a
regular eigenvector. We say $w$ is {\it $d$-regular} if the corresponding eigenvalue has
order $d$ ($d$ turns out to be independent of the choice of regular eigenvector).
Following \cite{rgly} we say
element $w$ is {\it $\mathbb Z$-regular} if $\langle w\rangle$ acts freely
on $\Delta$. It is proved in \cite[Prop. 1]{rgly} that $\mathbb Z$-regularity implies regularity. The converse does not hold in general. 

\begin{proposition}\label{regular} Suppose $w \in W\d$ is elliptic and
regular. Let $\theta \in \G_w$ such that $V(w, \theta)$ contains a
regular vector of $V$. Then $[n_w]=[\exp(\sqrt{-1} \theta\rho^\vee)\delta]$.
\end{proposition}

\begin{remark}
	The case where $w$ is $\mathbb Z$-regular is proved in \cite{rgly}*{Prop. 12}.
\end{remark}

\begin{proof} Let $\zeta=\exp(\sqrt{-1} \theta)$ and let
$\t=\delta \exp(\sqrt{-1} \theta \rho^\vee)=\delta \rho^\vee(\zeta)$. Let $d,m$ be
the orders of $\zeta \in \mathbb C^\times$ and $\t$ respectively.
Let $\xi$ be a primitive $m$-th root of unity such that
$\zeta=\xi^{m/d}$. We set $\iota=\tau^d=\delta^d$. Let $\fg, \ft$ be
the Lie algebras of $G$ and $T$ respectively. We denote by
$\fg^\iota$ and $\ft^\iota$ the subalgebras of $\iota$-fixed points
of $\fg$ and $\ft$ respectively. Then $\fg^\iota$ is also a
semisimple Lie algebra and $\ft^\iota$ is a Cartan subalgebra of
$\fg^\iota$.

	The automorphism $\t$ gives a periodic grading: $\fg=\oplus_{i
\in \mathbb Z/m\mathbb Z} \fg_i$, where $\fg_i$ is the
$\xi^i$-eigenspace of $\fg$ for $\t$. Then we have $\fg^\iota =
\oplus_{i \in \mathbb Z / d\mathbb Z} \fg^\iota_i$, where
$\fg^\iota_i=\fg_{im/d}$ is the $\zeta^i$-eigenspace for $\t$. This is
an $N$-regular periodic grading of ${\fg}^{\iota}$, see
\cite{panyushev}*{Section 3}.
	
Let $\fc' \subseteq \fg^\iota_1$ be a Cartan subspace. By the
construction in \cite{rgly}*{Subsection 3.1}, there is a $\t$-stable
Cartan subalgebra $\fs'$ of $\fg^\iota$ containing $\fc'$. Let $\fs$
be the centralizer of $\fs'$ in $\fg$. As $\fs'$ is conjugate to
$\ft^\iota$, $\fs$ is also a Cartan subalgebra of $\fg$ fixed by
$\t$. Let $g \in G$ such that $\ft=\Ad(g) \fs$, and set $\varepsilon=g
\t g\i$. Then $\varepsilon \in G \d$ fixes $\ft$ and lies in $N \d$,
where $N$ is the normalizer of $T$ in $G$. Let $\ft(n_w, \z)$ and
$\ft(\varepsilon, \z)$ be the $\z$-eigenspaces of $\ft$ for $n_w$ and
$\varepsilon$ respectively. Thanks to
\cite{springer_regular}*{Theorem 6.4 (iv)} and
the ellipticity of $w$, to show $n_w$ and $\varepsilon$ are conjugate,
it suffices to show $$\dim \ft(n_w, \z)= \dim \ft(\varepsilon, \z).$$
Notice that $\ft(\varepsilon, \z)=\Ad(g) \fc'$.
	
	Let $v \in V(w, \theta)$ be a regular point. We may assume
that $v$ is (strictly) dominant. Since $w^d(v)=v$ and $v$ is strictly
dominant, one has $w^d=\d^d=\iota$ and hence $w \in W^\iota \d$, where
$W^\iota$ is the subgroup of $\iota$-fixed points of $W$. Notice that
$\ft(n_w, \z)=\ft^\iota(n_w, \z)$ and that $W^\iota$ is the Weyl group
of $\ft^\iota$ in $\fg^\iota$, where $\ft^\iota(n_w, \z)$ is the
$\z$-eigenspace of $\ft^\iota$ for $n_w$. Then
\cite{springer_regular}*{Theorem 6.4 (ii)} says $$\dim \ft(n_w, \z) = \dim \ft^\iota(n_w, \z)=
a(d, \d),$$ where $a(d, \d)$ is defined in
\cite{springer_regular}*{Section 6} with
respect to $W^\iota$ and $\d$.
	
	On the other hand, applying \cite{panyushev}*{Theorem 3.3 (v)} to the
$N$-regular periodic grading $\fg^\iota = \oplus_{i \in \mathbb Z /
d\mathbb Z} \fg^\iota_i$, one deduces that $$\dim \ft(\varepsilon, \z)
= \dim \fc' = a(d, \d).$$ The proof is finished.
\end{proof}

\begin{corollary}
Suppose $w\in W\delta$ is elliptic and $d$-regular. Then
$$
[n_w]=[\exp(2\pi \sqrt{-1}\ch\rho/d)\delta].
$$
\end{corollary}

This follows from the fact that $V(w,2\pi/d)$ contains a regular vector.

\sec{The General Case}
\label{s:general}

We prove Theorem \ref{t:main}. We first collect some facts needed for the proof.

Suppose $w\in \Wext$ and the dominant chamber $\caC$ is in good position with
$w$ (see Section \ref{s:digression}). Let $w=\d x_1 x_2 \cdots x_k$ be the good factorization of $w$ as
in Proposition \ref{good-fac}, and define $\Delta_i,W_i,\ch\lambda_i$
$(0\le i\le k)$ as in \eqref{e:algorithm}(e-i).
Also define $\delta_0=\delta$ and 
$$
\delta_i=\delta x_1\dots x_i\quad(1\le i\le k).
$$
Recall (Proposition \ref{good-fac}) $\delta_i(\Pi_i)=\Pi_i$ $(0\le i\le k)$.

\begin{lemma} \label{dom}
For all $0\le i\le k$ we have
$\delta_i(\ch\lambda_i)=\ch\lambda_i$. In particular $\delta(\ch\lambda_0)=\ch\lambda_0$.
\end{lemma}

\begin{proof}
We proceed by downward induction on $i$. If $i=k$ then $\ch\lambda_i=0$ so there is nothing to prove.
Assume
  $\d_{i+1}(\l_{i+1}^\vee) = \l_{i+1}^\vee$ for some
  $0 \le i \le k-1$. We prove $\d_i(\l_i^\vee)=\l_i^\vee$. Since
  $\d_{i}(\Delta_i^+)=\Delta_i^+$, we
  have $$\d_{i}(\rho_i^\vee)=\rho_i^\vee.$$ Let $y_i \in W_i$ such
  that $\overline{\l_{i+1}^\vee} =
  y_i(\l_{i+1}^\vee)$. Then
  $$\overline{\l_{i+1}^\vee}=y_i \d_{i+1} y_i^{-1}
  (\overline{\l_{i+1}^\vee}) =y_i \d_i x_{i+1} y_i^{-1}
  (\overline{\l_{i+1}^\vee}).$$ Note that
  $y_i \d_i x_{i+1} y_i^{-1} \in W_i\delta_i$ and that
  $\overline{\l_{i+1}^\vee}$ is dominant for $\Delta_i^+$. Therefore,
  $\d_i$ fixes $\overline{\l_{i+1}^\vee}$ and by \eqref{e:algorithm}(i) also fixes
  $\l_i^\vee$ as desired.
\end{proof}

\begin{lemma} \label{average}
  Suppose $x \in W \d$ has order $d$.  Then  for $v\in V$, $n(x)\exp(2 \pi \sqrt{-1} v)$ is conjugate to
  $n(x) \exp(\frac{2 \pi \sqrt{-1}}{d}\sum_{i=0}^{d-1} x^i(v))$ by an element in $T$.
\end{lemma}
\begin{proof}
Let
$$
t=\exp(\frac{2\pi \sqrt{-1} }{d}\sum_{k=0}^{d-1} k x^{-k}(v))\in T.
$$
Then $t^{-1} n(x)\exp(2 \pi \sqrt{-1} v) t = n(x) \exp(2\pi \sqrt{-1}\frac{1}{d}\sum_{i=0}^{d-1} x^i(v))$ as desired.
\end{proof}

Consider the element $\delta_1=\delta x_1$. We may view this as an
automorphism of $W_1$ and the corresponding Levi subgroup $L_1$.  We
will need to conjugate $n(\delta_1)$ to an element of $T\delta$.
Note that $\d_1$ is not necessarily  an elliptic element of ${}^\d W$, so its liftings are not all conjugate in $G$.
For the purpose of our
argument, we fix a set of lifts $n(x) \in N(T)$ for $x \in {}^\d W$
such that

(1) $n(s_\a)^2=\exp(\pi \sqrt{-1} \a^\vee)$ for $\a \in \Pi$;

(2) $n(x x') = n(x) n(x')$ if $\ell(x x') = \ell(x) + \ell(x')$.

We need a formula for the cycle defined by these lifts.
\begin{lemma} \label{factor}{\cite{ls}*{Lemma 2.1A}}
  Given $x,y\in \Wext$ let
  $$
  S=\Delta^+\cap y\inv(\Delta^-)\cap y\inv x\inv(\Delta^+)
$$
and
$$
\ch\gamma=\sum_{\beta\in S}\ch\beta.
$$
Then
$$
n(x)n(y)=n(xy)\exp(\pi \sqrt{-1}\gamma^\vee). 
$$
\end{lemma}

\begin{lemma}
  \label{l:n(z)}
  Suppose $x,y\in\Wext$ satisfy $\ell(xy)=\ell(x)+\ell(y)$.
  Let $z=xyx\inv$.  Then
$$
n(z)=n(x)n(y)n(x)\inv \exp(\pi \sqrt{-1} \sum_{\beta\in \Delta^+\cap x(\Delta^-)\cap z\inv(\Delta^-)}\ch\beta).
$$
  
\end{lemma}

\begin{proof}
By Lemma \ref{factor}
$$
\begin{aligned}
  n(z)&=n(xyx\inv)\\
  &=n(xy)n(x\inv)\exp(-\pi \sqrt{-1}\sum_{\b \in \Delta^+\cap x(\Delta^-)\cap z\inv(\Delta^+)}\ch\beta)\\
  &=n(x)n(y)n(x\inv)\exp(-\pi \sqrt{-1}\sum_{\b \in \Delta^+\cap x(\Delta^-)\cap z\inv(\Delta^+)}\ch\beta)\\
  &=n(x)n(y)n(x\inv)\exp(\pi \sqrt{-1}\sum_{\b \in \Delta^+\cap x(\Delta^-)\cap z\inv(\Delta^-)}\ch\beta).\\
\end{aligned}
$$
The last equality follows from 
$n(x\inv)=n(x)\inv\exp(\pi \sqrt{-1}\sum_{\b \in \Delta^+\cap x(\Delta^-)}\ch\beta)$, which follows from another application of Lemma \ref{factor}.
\end{proof}


\subsection{A technical lemma}
\label{technical}

In the setting of the beginning of this section, we focus
our attention on $\Delta_1$, with simple roots $\Pi_1$,  and the element $\delta_1\in \Wext$
which preserves $\Pi_1$.
Recall that $L_1$ is the standard Levi subgroup corresponding to $\Pi_1$.

Recall (\ref{e:algorithm}(i)):
$$
\ch\l_1=\frac{d(\th_2-\th_1)}{2 \pi}
\ch\rho_1+\overline{\ch\l_2}.
$$

We say $v\in V^{\delta_1}$ is a general point of $V^{\delta_1}$ if $\langle\a, v\rangle=0$ for some
root $\a \in \Delta$ implies $\langle\a, V^{\d_1}\rangle=0$.  We can
find a general point of $V^{\delta_1}$ in a sufficiently small
neighborhood of $\ch\l_{1}$.  Furthermore, since the (open)
$\Pi_1$-dominant chamber intersects $V^{\delta_1}$ we may assume $v$
is strictly $\Pi_1$-dominant.
Given $v$ let  $z \in W$ be the (unique) minimal element such
that $z\i(v) = \overline v$, the unique dominant $W$-conjugate of $v$.
Then $z\inv(\ch\l_{1})=\overline{\ch\l_{1}}$.

Set $y=z\inv\delta_1z$.  Then $y(\overline v)=\overline v$, and
(since $\overline v$ is dominant and $\delta$ fixes the dominant
chamber) $\delta(\overline v)=\overline v$.

Set
$$
\Delta'=\{\alpha\in\Delta\mid \langle\alpha,V^y\rangle=0\}.
$$
Then $\Delta'$ is the root system of a Levi factor $L'$; we write $W_{\Delta'}$ for its Weyl group.
Note that, since $v$ is a general point of $V^{\delta_1}$,  $\Delta'=\{\alpha\in\Delta\mid \langle\alpha,\overline v\rangle=0\}$.
Then  (since $\delta(\overline v)=\overline v$), $\delta(\Delta')=\Delta'$. Also $y\in W_{\Delta'}\delta$,
and by the definition of $z$ we have

\begin{equation}
\label{e:lengthsadd}
\ell(zy)=\ell(z)+\ell(y).
\end{equation}

\begin{lemma} \label{tech}
Let $d$ be the order of $y$. Then
$$\Delta^- - z \i(\Delta_{1})-\Delta' = \cup_{i=0}^{d-1} y^{-i}(\Delta^- \cap z \i(\Delta^+) \cap z \i \d_1 \i(\Delta^-)).$$
Moreover, each root of $\Delta^- - z \i(\Delta_{1})-\Delta'$ lies in exactly $\frac{d \th_1} {2\pi}$ of the sets $y^{-i}(\Delta^- \cap z \i(\Delta^+) \cap z \i \d_1 \i(\Delta^-))$ for $0 \le i \le d-1$.
\end{lemma}
\begin{remark}
  The proof uses a similar method of counting root hyperplanes as in
  \cite{he_nie_minimal_finite}*{Lemma 2.1}.
\end{remark}

\begin{proof}
Let $v_0$ be a general point of $V(\d_1, \th_1)$.  Then $z \i(v_0)$ is a
general point of $V(y, \th_1)$. For $i \in \mathbb Z$, set
$v_i'=y^{-i} z^{-1} (v_0)$. Since $V(y, \th_1)_{\mathbb C}$ is the sum
of eigenspaces of $y$ of eigenvalues $e^{\pm i\th_1}$, the points
$v_i'$ for $i \in \mathbb Z$ are contained in the subspace $D$
spanned by $v_0'$ and $v_1'$. In particular $\dim D \le 2$.

We first consider the case where $\dim D=1$. In this case, we have
$\th_1 = \pi$ and hence $w=-1$. Then $\Delta_{1} = \emptyset$,
$\d_1 = w$, $\Delta' = \Delta$ and $z=1$. One checks that
$\Delta^- - z \i(\Delta_{1})-\Delta' = \emptyset = \Delta^- \cap
z \i(\Delta^+) \cap z \i \d_1 \i(\Delta^-)$. So the statement follows.

Now we consider the case where $\dim D=2$. Let $S \subseteq D$ the circle containing $v_i'$ for $i \in \mathbb Z$. Note that $v_i'=v_j'$ if $(i-j)\th_1/(2\pi) \in \mathbb Z$.

Set $\mathcal B= \Delta^- \cap z \i(\Delta^+) \cap z \i \d_1 \i(\Delta^-)$. Then \[\tag{a}\mathcal B=\{\b \in \Delta^-; H_\b \text{ separates } z^{-1}(\caC) \text{ from } \caC \text{ and } y^{-1} z^{-1}(\caC)\},\] where $\caC$ is the dominant chamber and $H_\b \subseteq V$ denotes the root hyperplane of $\b$.

As $z(\Delta')>0$, we see
$\Delta' \cap \mathcal B = \emptyset$. Since $y \in W' \d =\d W'$, we have $y^{-i}(\mathcal B) \subseteq \Delta^- - \Delta'$. Let $\b \in y^{-i}(\mathcal B)$. By (a), $H_\b$ separates $y^{-i} z^{-1} \caC$ from $y^{-i-1} z^{-1} \caC$. This means $\b \notin z^{-1}(\Delta_{1})$ as
$\delta_1(\Pi_1)=\Pi_1$.
So $$y^{-i}(\mathcal B) \subseteq \Delta^- - z^{-1}(\Delta_{1}) -\Delta'.$$

Let $\b \in \Delta^- - z^{-1}(\Delta_{1}) -\Delta'$. If $v_k' \in H_\b$ for some $k$, then $z^{-1}(V_{\th_1}) \subseteq H_\b$ since $v_k'=y^{-k} z^{-1}(v_0)$ is a general point of $z^{-1}(V_{\th_1})$. This means $\b \in z^{-1}(\Delta_{1})$, a contradiction. So $v_i' \notin H_\b$ for any $i \in \mathbb Z$. Thus, $H_\b \cap S = \{\pm p\}$ for some $p \in S$. Let $S^+, S^-$ be the two connected components of $S - \{\pm p\}=S - H_\b$ such that $\bar v, S^+$ are on the same side of $H_\b$.

Let $\overline\caC$ be the closure of $\caC$. Recall that $v$ is a general point of $V^{\d_1}$. As $\b \notin \Delta'$ and $\bar v \in \overline\caC$, $\bar v \notin H_\b$ and

\noindent(b) $y^{-i}(\bar v)=\bar v, y^{-i}(\caC)$ are on the same side of $H_\b$ for $i \in \mathbb Z$.

Similarly, one has

\noindent(c) $v_i', y^{-i} z^{-1}(\caC)$ are on the same side of $H_\b$ for $i \in \mathbb Z$.

Combining (a), (b) and (c), we have

\noindent(d) $\b \in y^{-i}(\mathcal B)$ if and only if $v_i' \in S^-$ and $v_{i+1}' \in S^+$.

Notice that the acute arc $(v_i', v_{i+1}')$ is of angle $0 < \th_1 < \pi$. It follows that that the number of integers $0 \le i \le d-1$ satisfying the condition (d) is exactly $d\th_1/2\pi$. The proof is finished.
\end{proof}

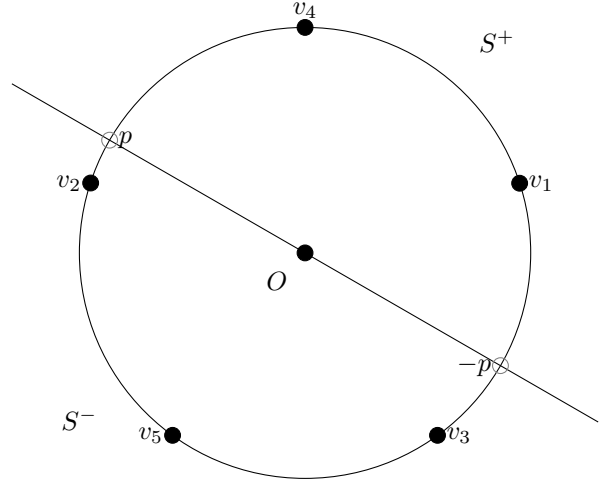
\begin{figure}
\center
\begin{tikzpicture}[scale=1.5]

 \draw[-] (-2.598,1.5) -- (2.598,-1.5) ;
  \path (1.7,1.7) node[above]{$S^{+}$};
  \path (-2,-1.3) node[below]{$S^{-}$};
 \draw (0,0) circle (2);
 \path (-0.25,-0.25) node(axis0){$O$};
 \draw[fill, black] (0,0) circle [radius=.07cm];
 \path (-1.732,1) node(text1)[right]{$p$};
 \draw[gray] (-1.732,1) circle [radius=.07cm];
\path (1.732,-1) node(text2)[left]{$-p$};
\draw[gray] (1.732,-1) circle [radius=.07cm];
\path (0,2) node(text3)[above]{$v_4$};
\draw[fill, black] (0,2) circle [radius=.07cm];
\path (1.902,0.618) node(text4)[right]{$v_1$};
 \draw[fill, black] (1.902,0.618) circle [radius=.07cm];
 \path (-1.9,0.62) node(text4)[left]{$v_2$};
 \draw[fill, black] (-1.9,0.62) circle [radius=.07cm];
 \path (1.174,-1.618) node(text4)[right]{$v_3$};
 \draw[fill, black] (1.174,-1.618) circle [radius=.07cm];
 \path (-1.174,-1.618) node(text4)[left]{$v_5$};
 \draw[fill, black] (-1.174,-1.618) circle [radius=.07cm];
\end{tikzpicture}

\caption{This is an illustration for the proof of Lemma \ref{tech}. Here $d=5$ and $\th_1 = \frac{4\pi}{5}$. The straight line is the intersection $H_\b \cap D$.}
\end{figure}

\begin{exampleplain}
  Let $W$ be of type $B_3$ and $\Pi=\{1, 2, 3\}$ such that $\a_3$ is the unique short simple root and $s_{13} = s_{31}$. Let $w=s_{12323}$. Then $\th_1 = \pi/2$, $d = 4$ and $w = \d_1 w_1$, where $\d_1 = s_{1232}$, $w_1 = s_3$, $\Pi_1=\{3\}$ and $V^{\d_1} = \mathbb R \a_3$. We take $v = \a_3$. Hence $z = s_{21}$, $y = s_{23}$ and $\Delta'=\{2, 3\}$. One checks that each root of $\Delta^- - z\i(\Delta_{1}) - \Delta'=\{-(\a_1), -(\a_1+\a_2), -(\a_1+\a_2+2 \a_3), -(\a_1+2 \a_2+2 \a_3)\}$ appears exactly one of the sets $\{-y^i(\a_1+\a_2)\}$ for $0 \le i \le 3$.
\end{exampleplain}

\subsection{Proof of Theorem \ref{t:main}}
We argue by induction on the cardinality of $\Delta$.  If
$\Delta=\emptyset$, then the statement is trivial.  Suppose that
$\Delta \neq \emptyset$ and that the statement is true for
any proper subsystem of $\Delta$.

Recall from the beginning of Section \ref{s:general} we have
$\delta_1=\delta x_1$, and this satisfies $\delta_1(\Pi_1)=\Pi_1$. Let
$L_1$ be the Levi subgroup of $G$ with root system $\Delta_1$, with derived group $L_{1,der}$.
The conjugation action of
$n(\d_1)$ preserves $L_{1,der}$. Then
$$
n(w),n(\delta_1) \in L_{1,der}\rtimes\langle n(\d_1)\rangle \subset {}^{\delta} G.
$$
Then $\Delta_1\subsetneq \Delta$.
Let $\ch\lambda_{0,L_1}$ be the element defined for $L_1$ with the angle sequence $\theta_2,\theta_3,\dots$,
i.e.
\begin{equation}
\label{e:lambda0L1}
\ch\lambda_{0,L_1}=\frac{d\theta_2}{2\pi}\ch\rho_1+\overline{\ch\lambda_2}
\end{equation}
Then by the inductive hypothesis
applied to $L_{1, der}$ we have
\begin{subequations}
\renewcommand{\theequation}{\theparentequation)(\alph{equation}}  
\label{e:maincalc}
\begin{equation}
n(w)\sim_G n(\d_1) \exp(\frac{2 \pi \sqrt{-1} \ch\l_{0, L_1}}{d}),
\end{equation}
where $\ch\l_{0, L_1}$ is given by \eqref{e:lambda0L1}.
Define $y$ and $z$ as at the beginning of Section \ref{technical}. Then \eqref{e:lengthsadd}
holds so we can apply  Lemma \ref{l:n(z)} to conclude
\begin{equation}
\begin{aligned}
  n(w)&\sim_G n(z)^{-1} n(\d_1) \exp(\frac{2 \pi \sqrt{-1} \ch\l_{0, L_1}}{d}) n(z)\\
  &=n(y) t \exp(\frac{2 \pi \sqrt{-1} z^{-1}(\ch\l_{0, L_1})}{d}),
\end{aligned}
\end{equation}
where

\begin{align}
  t&=\exp (\pi \sqrt{-1}\sum_{\b \in \Delta^+ \cap z(\Delta^-) \cap \d_1 \i(\Delta^-)} z^{-1}(\b^\vee))\\
  &=
\exp (\pi \sqrt{-1}\sum_{\b \in \Delta^- \cap z\inv(\Delta^+) \cap z\inv\d_1 \i(\Delta^-)} \b^\vee).
\end{align}
\end{subequations}

As in Section \ref{technical} let $\Delta'=\{\alpha\in\Delta\mid \langle\alpha, V^y\rangle=0\}$.
This is the root system of a Levi factor $L'$, with Weyl group $W_{\Delta'}$.

\begin{lemma}\label{empty}
	We have $z\i(\Delta_1^+) \subseteq \Delta^+$ and $z\i(\Delta_1) \cap \Delta' = \emptyset$.
\end{lemma}
\begin{proof}
  Notice that $\Delta' = \{\a \in \Delta, \langle\a, \bar v\rangle = 0 \}$. For $\gamma \in \Delta_1^+$ we have $$\langle z\i(\gamma), \bar v\rangle = \langle z\i(\gamma), z\i(v)\rangle = \langle \gamma, v\rangle > 0,$$ where the last inequality follows from
  the fact that $v$ is strictly $\Pi_1$-dominant. So $z\i(\gamma) > 0$ and $z\i(\gamma) \notin \Delta'$ as desired.
\end{proof}


Let $\ch\gamma$ be the summand in (d). By Lemma \ref{tech} and Lemma \ref{empty}, with $c=\frac{d\theta_1}{2\pi}\in\Z$ we have
\begin{subequations}
\label{e:maincalc1}
\begin{equation}
\sum_{j=0}^{d-1}y^j\ch\gamma=2c(-\rho^\vee+z^{-1}(\rho_1^\vee)+\rho_{L'}^\vee).
\end{equation}
Since (d) is unchanged if we replace $\ch\gamma$ with $-\ch\gamma$, 
by  Lemma \ref{average}  we see $n(y)t$ is $T$-conjugate to
\begin{equation}
n(y)\exp(\frac{2\pi \sqrt{-1}c(\ch\rho-z\inv(\ch\rho_1)-\ch\rho_{L'})}{d})
\end{equation}
and therefore
\begin{equation}
n(w)\sim n(y)\exp(\frac{2\pi \sqrt{-1}c(\ch\rho-z\inv(\ch\rho_1)-\ch\rho_{L'})}{d})
\exp(\frac{2 \pi \sqrt{-1}z\inv( \ch\l_{0, L_1})}{d}).
\end{equation}
\end{subequations}
\begin{lemma}
No roots in $\Delta'$ vanish identically on $V(y, \th_1)$.
\end{lemma}

\begin{proof}

Let $\a \in \Delta'$. We have to show that
$\langle\a, V(y, \th_1)\rangle \neq \{0\}$. 
Assume otherwise. Noticing that
$\d_1=w (x_2 \cdots x_k) \i$ and $x_2 \cdots x_k \in W_1$ fixes each
point of $V(w, \th_1)$, we have
$V(w, \th_1) \subseteq V(\d_1, \th_1)=z^{-1}(V(y, \th_1))$. Since $\langle\alpha,V(y,\theta_1)\rangle=\{0\}$, we conclude that
$z(\a) \in \Delta_1$, which contradicts Lemma \ref{empty}.
\end{proof}

Therefore we can apply Proposition \ref{regular} to conclude
$n(y)\sim_{L'}\delta\exp(\frac{2\pi \sqrt{-1}c}d\ch\rho_{L'})$.
Using the fact that $\ch\rho-z\inv\ch\rho_1-\ch\rho_{L'}$ and $z\inv(\ch\lambda_{0,L_1})$ are central in $L'$, we can insert
this in \eqref{e:maincalc1}(c)  to give
$$
\begin{aligned}
n(w)&\sim \delta\exp(\frac{2\pi \sqrt{-1}c}d\ch\rho_{L'})
  \exp(\frac{2\pi \sqrt{-1}c(\ch\rho-z\inv(\ch\rho_1)-\ch\rho_{L'}}d)
  \exp(\frac{2 \pi \sqrt{-1}z\inv( \ch\lambda_{0,L_1})}{d})\\
  &=\delta\exp(\frac{2\pi \sqrt{-1}(c\ch\rho+z\inv(\ch\lambda_{0,L_1}-c\ch\rho_1))}d)\\
\end{aligned}
$$
By \eqref{e:lambda0L1} and the definition of $c$ we compute:
$$
z\inv(\ch\lambda_{0,L_1}-c\ch\rho_1)=z\inv(\overline{\ch\lambda_2}+\frac{d\theta_2}{2\pi}\ch\rho_1-\frac{d\theta_1}{2\pi}\ch\rho_1)
=z\inv(\ch\lambda_1)=\overline{\ch\lambda_1}.
$$
Inserting this we conclude
$$
n(w)\sim\delta\exp(\frac{2\pi \sqrt{-1}(c\ch\rho+\overline{\ch\lambda_1})}d)=
\delta\exp(\frac{2\pi \sqrt{-1}}d\ch\lambda_0).
$$
This completes the proof of Theorem \ref{t:main}.
\qed

\sec{Kac diagrams}
\label{s:kac}

We consider a connected complex semisimple group $G$, with Cartan subgroup $T$, root system $\Delta$,
simple roots $\Pi$ and Weyl group $W$. We are also given
an  automorphism $\delta$ of $G$, possibly trivial, of finite order $n$, preserving a pinning.
The semisimple
conjugacy classes of $G\delta$ of finite order are parametrized by their
{\it Kac diagrams}. We summarize the statements here, and refer to \cite{DM}, \cite{kac_book}, 
\cite{ov} and \cite{reeder_torsion} for details.

Let $X^*$, respectively $X_*$, be the lattice of characters (resp. co-characters) of $T$.
Let $R=\Z\langle\Delta\rangle\subset X^*$ be
the root lattice, and $\ch\Delta\subset \ch R\subset X_*$ the co-roots and co-root lattice.

Recall that $V=X_*\otimes\Q$. Then $\delta$ acts on $V$, and we also write
$\delta$ for the transpose action on $V^*=\Hom_\Q(V,\Q)$.
We write $\delta$ as a superscript to denote fixed points.
We identify $(V^\delta)^*$ with $(V^*)^\delta$, and we view this as a subset of $\Lie(T)$.

We can write $T=(T^\delta)^0(1-\delta)T$, where $\phantom{}^0$ denotes identity component,
and $(1-\delta)T=\{t\delta(t\inv)\mid t\in T\}$. Both groups are (connected) tori.
It follows easily from this that every element of $G\delta$ is $G$-conjugate
to an element of $(T^\delta)^0$.

Define
\begin{equation}
\label{e:e}
e(\ch\gamma)=\exp(2\pi \sqrt{-1}\ch\gamma) \quad(\ch\gamma\in V^\delta).
\end{equation}
This map is surjective  onto the elements of finite order in $(T^\delta)^0$.

Although we will make no use of this fact, it is interesting to note that if $G$ is simple then
$G^\delta$  and $T^\delta$ are connected.

Let $\Pi_\delta$ the set of orbits of $\delta$ on $\Pi$. The map
$\Pi\ni\alpha\mapsto \alpha|_{T^{\delta}}$ identifies $\Pi_\delta$ with 
a set of simple roots of 
$G^{\delta}$. Write $\Pi_\delta=\{\alpha_1,\dots,\alpha_r\}$.

Let $\Delta_\delta$ be the root system of
$(T^{\delta})^0\subset (G^{\delta})^0$, and let
$\ch\Delta_\delta=\{\ch\alpha\mid \alpha\in\Delta_\delta\}\subset
V^\delta$ be the canonical coroot defined by
(\cite{bourbaki_4-6}*{Chapter 6, Section 1.1}).
Set $\ch R_\delta=\Z\langle\ch\Delta_\delta\rangle$, the coroot lattice of $\Delta_\delta$.
Then
$$
X_*((T^{\delta})^0)=(X_*)^\delta\quad\text{and}\quad \ch R_\delta=(\ch R)^\delta.
$$
Let $W_\delta$ be the Weyl group of $\Delta_\delta$.

Define projection $P:V\rightarrow V^\delta$ by
$P(v)=\frac1n \sum_{i=0}^{n-1}\delta^iv$.
Then $P(X_*)$ is a lattice containing $X_*^\delta$ of finite index,
and $W_\delta$ acts on $P(X_*)$ and $P(\ch R)$. Define:
\begin{equation}
\label{e:affine}
\begin{aligned}
\wh W_\delta&=W_\delta\ltimes P(X_*), \\
\wt W_\delta&=W_\delta\ltimes P(\ch R).
\end{aligned}
\end{equation}
Then $\wt W_\delta$ is an affine Weyl group. Let $\overline C$ be a
fundamental domain for the action of $\wt W_\delta$ on $V^\delta$.
Furthermore
$\wh W_\delta$ is an extended affine Weyl group, and
$$
\Omega=\wh W_\delta/\wt W_\delta\simeq P(X_*)/P(\ch R)
$$
is a finite group which acts naturally on $\overline C$.

\begin{lemma}
\label{l:Omega}
If $G$ is adjoint then $\Omega\simeq \pi_1(G^\delta)$.
\end{lemma}

\begin{proof}
If $\delta=1$ this is standard. In general since $G$ is adjoint it
reduces easily to the simple case and then a case-by-case check. In
fact in the twisted cases $\Omega$ is trivial except for $\twoAodd$ and $\twoD$, in which case it has order $2$.
\end{proof}

\begin{lemma}
\label{l:kac1}
Every semisimple conjugacy class of finite order in $G\delta$ is of
the form $[e(\ch\gamma)\delta]$ for some $\ch\gamma\in V^\delta$.
  
In particular suppose $\ch\gamma,\ch\tau\in V^\delta$.  Then $[e(\ch\gamma)\delta]=[e(\ch\tau)\delta]$ if and only if there exists $w\in \wh W_\delta$
such that $w\ch\gamma=\ch\tau$.

Suppose $\ch\gamma,\ch\tau\in \overline C$. Then
$[e(\ch\gamma)\delta]=[e(\ch\tau)\delta]$ if and only if there exists
$\omega\in \Omega$ such that $\omega(\ch\gamma)=\ch\tau$.

\end{lemma}

\begin{proof}[Sketch of proof]
We've already discussed the first assertion. The second follows from
the calculation that
$$
e(\ch\mu)e(\ch\gamma)\delta e(-\ch\mu)=e((1-\delta)\ch\mu+\ch\gamma)\delta\quad (\ch\mu\in V,\ch\gamma\in V^\delta)
$$
and the fact that
$$
[(1-\delta)V+X_*]^\delta=P(X_*).
$$
The final assertion is standard.
\end{proof}

We now apply the  theory of  affine Weyl groups to describe $\overline C$
when $G$ is simple.
If $\zeta$ is a root of unity let
$$
\g[\zeta]=\{X\in \g\mid \Ad(g)(X)=\zeta X\}.
$$
Then $\g[\zeta]$ is $T^\delta$ invariant. Let $\alpha_0$ be the lowest weight of $T^\delta$ acting on $\g[\zeta]$.
Then
$$
-\alpha_0=
\begin{cases}
  \text{highest root of }\Delta_\delta, &\text{ if } \delta=1;\\
  \text{highest short root of }\Delta_\delta, &\text{ if }\delta\ne 1,\Pi^\delta\ne\emptyset;\\
  \text{2*highest short root of }\Delta_\delta, &\text{ if }\delta\ne 1,\Pi^\delta=\emptyset.
\end{cases}
$$
The last case occurs if $G$ is of type $A_{2n}$ and $\delta$ has order $2$ (type $\twoAeven$).
Set $\wt\Pi_\delta=\{\alpha_0,\alpha_1,\dots,\alpha_r\}$,
and let $\ch\gamma_1,\dots,\ch\gamma_i\in V^\delta$ be the corresponding fundamental co-weights.
Define integers $c_0,c_1,\dots, c_r$ by $c_0=1$ and 
$$
\sum_{i=0}^r c_i\alpha_i=0.
$$
We define the {\it affine Dynkin diagram}  $\Daffine$ of $(G,\delta)$ to be the
Dynkin diagram of $\wt\Pi_\delta$. We equip each node with its label $c_i$.
See \cite{ov}*{Reference Chapter,Table 6} for a list of these diagrams.

The automorphism group of $\Daffine$ is isomorphic to the
automorphism group of $\overline C$. The group $P(X_*)$ acts by
translation on $V^\delta$, and this induces an action of $\Omega$ on
$\overline C$ and $\Daffine$.

Let $\wt\alpha_0$ be the affine function
$$
\wt\alpha_0=\alpha_0+\frac1n.
$$
We define the {\it affine coordinates} of  $\ch\gamma\in V^\delta$ to be
$$
(\wt\alpha_0(\ch\gamma),\alpha_1(\ch\gamma),\dots, \alpha_r(\ch\gamma)).
$$
The affine coordinates $(a_0,\dots, a_r)$ of a point  in $V^\delta$ satisfy
$\sum_{i=0}^r c_ia_i=\frac1n$.

For the fundamental domain $\overline C$ we take
points whose affine coordinates $(a_0,\dots, a_r)$ satisfy $a_i\ge 0$ $(0\le i\le r)$.

\begin{definition}
A Kac diagram for $(G,\delta)$ is a vector $\mathcal D=[a_0,\dots, a_r]$ where
each $a_i$ is a non-negative integer and $\text{GCD}(\{a_0,\dots, a_r\})=1$.
Set $d(\D)=\sum_{i=0}^r a_ic_i$, $n=\text{order}(\delta)$, and define
$$
e(\D)= e(\frac n{d(\D)}\sum_{i=1}^r a_i\ch\gamma_i)\delta.
$$
\end{definition}
Note that $\frac n{d(\D)}\sum_{i=1}^r a_i\ch\gamma_i\in \overline C$.
Here is the conclusion.

\begin{proposition}
  \label{p:kac}
Suppose $g\delta\in G\delta$ satisfies $g^d\in Z(G)$. Then there is a Kac diagram $\D$,
with $d(\D)=d$, such that $[g\delta]=[e(\D)]$.

If $\D,\mathcal E$ are Kac diagrams then $[e(\D)]=[e(\mathcal E)]$ if and only if
there exists $\omega\in \Omega$ satisfying $\omega(\D)=\mathcal E$.
\end{proposition}

We examine the role of the group $\Omega$ more closely.  If
$z\in Z(G^\delta)$ the map $[g\delta]\rightarrow [zg\delta]$ is a well
defined map of conjugacy classes in $G\delta$. Via the Proposition
this induces an action of $Z(G^\delta)$ on Kac diagrams.

Note that the Kac diagrams for $G$ are independent of isogeny; the
only role that isogeny plays is in the action of $\Omega$.  So suppose
$\D$ is a Kac diagram. View it as giving a conjugacy class of finite
order in $\Gsc\delta$ where $\Gsc$ is the simply connected cover of
$G$.  Thus $Z((\Gsc)^\delta)$ acts on Kac diagrams. Recall
$\Omega\simeq \pi_1((\Gsc)^\delta)$, which is a quotient of
$Z((\Gsc)^\delta)$.
This is compatible with Proposition \ref{p:kac}: if $z\in Z((\Gsc)^\delta)$ and $\D$ is a Kac diagram then
$$
[e(z\D)] = [p(z)e(\D)],
$$
where $p:Z((\Gsc)^\delta)\rightarrow Z(G^\delta)$. 

\begin{lemma}\label{6.7}
The orbits of  $Z((\Gsc)^\delta)$  and  $\Aut(\Daffine)$  on  the nodes of $\Daffine$
are the same.
\end{lemma}

\begin{proof}
The nodes with label $1$ are in bijection with $Z((\Gsc)^\delta)$,
so the action of $Z((\Gsc)^\delta)$ on these nodes is simply
transitive and the result is immediate.
The remaining nodes follow from a case--by--case check.
\end{proof}

\begin{proposition}
\label{p:fixed}
Suppose $w\delta\in W\delta$ is elliptic, and $n(w\delta)\in G\delta$ is a representative of $w$.
Then the Kac diagram of $n(w\delta)$ is fixed by $\Aut(\Daffine)$.
\end{proposition}

\begin{proof}
Suppose $w\delta\in W\delta$, with representative
$g\delta\in G\delta$.  If $z\in ZG^\delta$ then $zg\delta$ is also a
representative of $w\delta$, so if $w\delta$ is elliptic
then $[zg\delta]=[g\delta]$.  Therefore the Kac diagram of $[g\delta]$ is
fixed by $Z(G^\delta)$, hence by $Z((\Gsc)^\delta)$ (which acts by
projection), and hence by $\Aut(\Daffine)$ by 
Lemma \ref{6.7}.
\end{proof}

For example in (untwisted) type $A_n$ every node has label $1$, so the
only Kac diagram which is fixed by all automorphisms has all labels
$1$, which corresponds to the Coxeter element. This proves the well-known fact
that this is the the only elliptic conjugacy class in this case.

\sec{Some applications to elliptic conjugacy classes}
In this section, we make a digression and discuss some applications to elliptic conjugacy classes of finite Weyl groups. 

\begin{proposition}\label{stable-d'}
Assume that $W$ is irreducible. Then every elliptic $W$-conjugacy class of $W\delta$ is stable under the diagram automorphisms of $W$
which commute with $\delta$.
\end{proposition}

\begin{remark}
  In fact, the result is true for any finite Coxeter group (see \cite{geck_pfeiffer}*{Theorem 3.2.7} when $\d=id$ and
  \cite{he_minimal_length_double_cosets}*{Theorem 7.5} in general). The original proof (even for finite Weyl groups) is based on a characterization of elliptic conjugacy classes via the characteristic polynomials and length functions. Such characterization is established via a laborious case-by-case analysis on the elliptic conjugacy classes with the aid of computer for exceptional groups. Now we give a general proof for finite Weyl groups via the Kac diagram. 
\end{remark}

\begin{proof}
  This follows easily from Proposition \ref{p:fixed}. If $\tau$ is an
  automorphism of $W$, commuting with $\delta$, it induces an
  automorphism of $\Wext$, preserving $W\delta$. Let $G$ be the corresponding simply connected group.
  Then $\tau$ lifts to an automorphism, also denoted $\tau$, of $G$.
  If $w\delta\in W\delta$ is elliptic, so is $\tau(w\delta)$, and
  $[\tau(w\delta)]=[w\delta]$ if and only if $[\tau(n(w\delta))]=[n(w\delta)]$, where $n(w\delta)$ is a representative
  in $G\delta$ of $w\delta$, i.e. $\tau(n(w\delta))$ and $n(w\delta)$ have the same Kac diagram.
  Now the result follows from Proposition \ref{p:fixed}.
\end{proof}

Proposition \ref{stable-d'} is used in an essential way to prove that in finite Weyl groups, elliptic conjugacy classes never fuse. 

\begin{theorem}\label{fuse}
	Let $W$ be a finite Weyl group and $\mathcal O$ be a $W$-conjugacy class of $W \d$. Let $J \subset I$ with $\d(J)=J$ and $\mathcal O \cap W_J \d$ contains an elliptic element of $W_J \d$. Then $\mathcal O \cap W_J \d$ is a single conjugacy class of $W_J$. 
\end{theorem}

This result was first proved in \cite{geck_pfeiffer}*{Theorem 3.2.11} when $\d=id$ and in
\cite{CH}*{Theorem 2.3.4} in general. The strategy is to first reduce to the case where $W$ is irreducible, then to reduce to the case where $W_J$ is irreducible. Note that the different $W_J$-conjugacy classes in $\mathcal O \cap W_J \d$ are obtained from one another by diagram automorphisms of $W_J$. The final (and crucial) step in \cite{geck_pfeiffer} and \cite{CH} was to use the characterization of elliptic conjugacy classes to deduce that the intersection is a single $W_J$-conjugacy class. Now the final step may be replaced by Proposition \ref{stable-d'}, the proof of which is simpler than the characterization of elliptic conjugacy classes. 

\sec{Non-Elliptic elements}
\label{s:nonelliptic}

For $w \in W$, we denote by $\text{supp}(w)$ the support of $w$, i.e., the set of simple reflections that occur in some (or equivalently, any) reduced expression of $w$. We define $$\text{supp}(w\delta): =\bigcup_{i \in \mathbb Z} \delta^i(\text{supp}(w)).$$ 

By \cite{geck_pfeiffer}*{\S 3.1} and
\cite{he_minimal_length_double_cosets}*{\S 7}, a conjugacy class of $\Wext$ is
elliptic if and only if it does not intersect with
$\Wext_J=W_J \rtimes \langle \delta \rangle$ for any proper
$\delta$-stable subset $J$ of $\Pi$, in other words,
$\text{supp}(w)=\Pi$ for any $w$ in the conjugacy class.

We have the following result (see \cite{geck_pfeiffer}*{Corollary
3.1.11} for untwisted conjugacy classes and \cite{CH}*{Proposition
2.4.1} in the general case).

\begin{proposition}\label{x-delta}
	Let $w_1, w_2 \in W\delta$ be minimal length elements in the same conjugacy class. Let $J_i=\text{supp}(w_i)$ for $i=1,2$. Then there exists $x \in {}^{J_2} W^{J_1} \cap W^\delta$ with $x J_1 x^{-1}=J_2$. 
\end{proposition}

Consider the set $\mathcal P_{\delta}$ of pairs $(J, D)$, where $J \subset \Pi$ is a $\delta$-stable subset and $D \subset W_J \delta$ is an elliptic conjugacy class of $\Wext_J$. The equivalence relation on $\mathcal P_{\delta}$ is defined by $(J, D) \sim (J', D')$ if there exists $x \in {}^{J'} W^J \cap W^\delta$ such that $x J x^{-1}=J'$ and $x D x^{-1}=D'$. 

Combining Theorem \ref{fuse} with Proposition \ref{x-delta}, we have 

\begin{theorem}\label{cal-P}
	The map $$C \mapsto \{(J, C \cap W_J \delta)\mid J=\text{supp}(w) \text{ for some } w \text{ of minimal length in } C\}$$ induces a bijection from $[W\delta] \to \mathcal P_{\delta}/\sim$. 
\end{theorem}

For untwisted conjugacy classes, the statement is obtained by Geck and
Pfeiffer in \cite{geck_pfeiffer}*{Theorem 3.2.12}. The general case is
proved in a similar way.

Let $C\in [\Wext]$ and $w \in C$. In general the lifts of $w$ to $\Gext$ are not $G$-conjugate. However there is a reasonable canonical choice of this lifting, defined as follows.

Without loss of generality, we assume that $C \subset W \delta$. Let $(J, D) \in \mathcal P_{\delta}$ be an element corresponds to $C$. Let $L_J$ be the standard Levi subgroup corresponds to $J$.  Apply the algorithm of Section \ref{const} to $({}^\delta L_J, D)$ to construct a conjugacy class $\Psi_J(D)$ in ${}^\delta L_J$, and thus (by acting by $G$) a conjugacy class $\widetilde{\Psi_J(D)}$ in $\Gext$.

\begin{proposition} The map $$[W \delta] \to [\Gext^{\ss}], \qquad C \mapsto \widetilde{\Psi_J(D)}$$ is well-defined. 
\end{proposition}

\begin{proof}
	Let $(J, D), (J', D')$ be elements in $\mathcal P_{\delta}$ that correspond to $C$. By Theorem \ref{cal-P}, there exists $x \in {}^{J_2} W^{J_1} \cap W^\delta$ with $x J_1 x^{-1}=J_2$. As discussed in \cite{AH}*{Section 2} the Tits group provides a section $\sigma:W\rightarrow N(T)$ satisfying
	$\delta(\sigma(w))=\sigma(\delta(w))$. In particular, $\sigma(x)$ is $\delta$-stable. Since $x J x^{-1}=J'$, we have $\sigma(x) L_J \delta(\sigma(x))^{-1}=\sigma(x) L_J \sigma(x)^{-1}=L_{J'}$. Since $x D x^{-1}=D'$, we have $\sigma(x) \Psi_J(D) \sigma(x)^{-1}=\Psi_{J'}(D')$. Hence $\widetilde{\Psi_J(D)}=\widetilde{\Psi_{J'}(D')}$. 
\end{proof}

\sec{Tables}
\label{s:tables}

For each exceptional group we list representatives of the elliptic conjugacy classes in $W$,
their Kac diagrams, and some other information.

We use the Bourbaki numbering of the simple roots \cite{bourbaki_4-6}.
Each table is preceded by the affine Dynkin diagram with the labels of the nodes.

\begin{enumerate}
\item Name: name of the elliptic conjugacy class, as in \cite{carter_conjugacy_classes} and \cite{geck_pfeiffer}.
\item d: order of the elements in the conjugacy classes.
\item Kac diagram: with respect to the given affine Dynkin diagram.
\item Centralizer: type of the derived group of the centralizer of the nilpotent element.
\item good: $w^d$ in the braid monoid (see Theorem \ref{t:braid}). Here $\Delta_S$ is the long element of the
  Weyl group $W_S$ of the Levi factor defined by $S$, and $\Delta$ is the long element of the Weyl group of $W$.
\end{enumerate}

These tables were computed using the algorithm of Section
\ref{const}.

Alternatively one can compute the Kac diagram in many cases using
standard techniques, starting with the result for regular elements.
The remaining cases require a number of case-by-case arguments, for
example see \cite{rgly}*{Section 8}. This is how the tables in
\cite{bouwknegt}, \cite{rgly}, and \cite{levy_exceptional} were
computed

\begin{exampleplain}
Consider the conjugacy class $E_8(a_7)$ of  $W(E_8)$
\cite{geck_pfeiffer}*{Table B.6}. We take the following
representative
$$
w=2343654231435426543178
$$
of order $d=12$ and length $22$. Let $\zeta$ be a primitive $12^{th}$ root of unity.
The eigenvalues of $w$ are $\{\zeta^k\mid k=1,2,5,7,10,11\}$.
The dimension of the eigenspaces are $1,2,1,1,2,1$, respectively.
In the notation of the algorithm we have
$$
\Gamma_w=\{\theta_1,\theta_2,\theta_3\}=\{2\pi /12, 4\pi /12, 10\pi /12\}=\frac{2\p}{12}*\{1,2,5\}.
$$
Note that
$$
\frac{d(\theta_2-\theta_1)}{2\pi}=\frac{d(\theta_1-\theta_0)}{2\pi}=1.
$$
We have
$$
0=F_0\subset F_1\subset F_2\subset F_3=V,
$$
where the $F_i$ have dimensions $0,2,6$ and $8$, respectively.

In particular $F_1=V(w,2\pi/12)$ is two-dimensional.  The set $\Delta_1$ of roots vanishing on this space
is a standard Levi subgroup of type $D_4$, with simple roots $\{2,3,4,5\}$.
It turns out that $\Delta_2=\emptyset$, so we have

  $$
  \Delta=\Delta_0=E_8\supset \Delta_1=D_4\supset \Delta_2=\Delta_3=\emptyset.
$$

The algorithm gives the following elements in turn:
$$
\ch\lambda_3=\ch\lambda_2=0,
$$
$$
\ch\lambda_1=\frac{d(\theta_2-\theta_1)}{2\pi}\ch\rho_1=\ch\rho_1.
$$
Next find $w\in W(\Delta_0)=W$ so that $w\ch\rho_1$ is dominant, and then set
$$
\ch\lambda_0=\frac{d(\theta_1-\theta_0)}{2\pi}\ch\rho+w\ch\rho_1=\ch\rho+w\ch\rho_1.
$$
In fundamental weight coordinates we have
$$
\begin{aligned}
  \ch\rho_1&=(-3,1,1,1,1,-3,0,0),\\
  w\ch\rho_1&=(0,0,0,0,0,0,1,1),\\
  \ch\lambda_0=\ch\rho+w\ch\rho_1&=  ( 1,1,1,1,1,1,2,2),\\
  \ch\lambda_0/12&=(  1,1,1,1,1,1,2,2)/12.
\end{aligned}
$$
This element is dominant but not in the fundamental alcove; its affine coordinates are
$$
(1,1,1,1,1,1,2,2,-22)/12.
$$
Applying the affine Weyl group takes this to the element $(0,0,1,0,1,0,0,1)/12$,
or affine coordinates  $(0,0,1,0,1,0,0,1,1)/12$. The corresponding  Kac coordinates are therefore $(0,0,1,0,1,0,0,1,1)$.
Note that the sum of the coefficients times the corresponding labels is $1*4+1*5+1*2+1*1=12$.
See the corresponding line in the $E_8$ table. 
Compare \cite{rgly}*{Section 8}.
\end{exampleplain}

\begin{remarkplain}
In \cite{geck_pfeiffer}*{Table B.6} there is a different
representative for this conjugacy class.
Although this representative is good,
it turns out the positive chamber is not in good position for this element
(both ``good'' and ``good position'' are defined in Section \ref{s:digression}), and
in particular the Levi subgroups defined by \eqref{e:levis} are not standard.
\end{remarkplain}

\newpage

$$
G_2:\quad \begin{dynkinDiagram}[Kac,extended,root radius=.12cm, edge length=1.0cm]{G}{2}
\dynkinLabelRoots{1,2,3}
\end{dynkinDiagram}
$$

\medskip

\hskip.8in
\begin{tabular}{|l|l|l|l|l|l|}
\hline
w&   d    &  Kac diagram &good&Centralizer\\\hline
$12$ & $6$ &   $111$ & $\Delta^2$&$*$ \\\hline
$1212$ & $3$ &  $110$ & $\Delta^2$&$A_1$ \\\hline
$w_0$ & $2$ &  $010$ & $\Delta^2$&$2A_1$ \\\hline
  \end{tabular}

  \bigskip  \bigskip

$$
^3D_4:\quad \begin{dynkinDiagram}[Kac,extended,root radius=.12cm, edge length=1.0cm]{D}[3]{4}
\dynkinLabelRoots{1,2,1}
\end{dynkinDiagram}
$$

\medskip

\hskip.7in
\begin{tabular}{|l|l|l|l|l|l|l|l|}
\hline
w&   d    &  Kac diagram &good&Centralizer\\\hline
  $12$ & $12$ &   $111$ & $\Delta^2$&$*$ \\\hline
  $132132$ & $6$ &   $010$ & $\Delta^2\Delta_{23}^4$&$2A_1$ \\\hline
  $1323$ & $6$ &   $101$ & $\Delta^2$&$A_1$ \\\hline
      $13213423$ & $3$ &   $001$ & $\Delta^2$&$A_2$ \\\hline
  \end{tabular}

\newpage

$$
F_4:\quad \begin{dynkinDiagram}[Kac,extended,root radius=.12cm, edge length=1.0cm]{F}{4}
\dynkinLabelRoots{1,2,3,4,2}
\end{dynkinDiagram}
$$

\medskip
\setlength\extrarowheight{3pt}

\hskip.5in
\begin{tabular}{|l|l|l|l|l|}
\hline
Name&   d    &  Kac diagram &good&Centralizer\\\hline
$F_4$ & $12$ &   $11111$ & $\Delta^2$&$*$ \\\hline
$B_4$ & $8$ &   $11101$ & $\Delta^2$&$A_1$ \\\hline
$F_4(a1)$ & $6$ &  $10101$ & $\Delta^2$&$2A_1$ \\\hline
$D_4$ & $6$ &  $11100$ & $\Delta^2\Delta_{34}^4$&$A_2$ \\\hline

$C_3+A_1$ & $6$ &   $01010$ & $\Delta^2\Delta_{12}^4$&$3A_1$ \\\hline
$D_4(a1)$ & $4$ &   $10100$ & $\Delta^2$&$A_1+A_2$ \\\hline
$A_3+\wt A_1$ & $4(8)$ &   $02010$ & $\Delta^2\Delta_{23}^2$&$3A_1$ \\\hline
$A_2+\wt A_2$ & $3$ &   $00100$ & $\Delta^2$&$2A_2$ \\\hline
$4A_1$ & $2$ &   $01000$ & $\Delta^2$&$A_1+C_3$\\\hline
\end{tabular}

\medskip

The conjugacy class $A_3+\wt A_1$ of $W$ has order $4$, but its lift
to a semisimple conjugacy class has order $8$ \cite{AH}*{Theorem B}.

\bigskip\bigskip

$$
^2E_6:\quad \begin{dynkinDiagram}[Kac,extended,root radius=.12cm, edge length=1.0cm]{E}[2]{6}
\dynkinLabelRoots{1,2,3,2,1}
\end{dynkinDiagram}
$$

\medskip
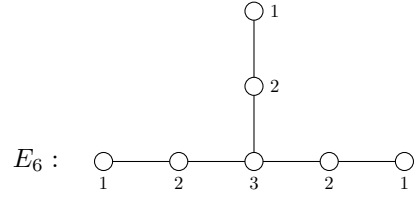
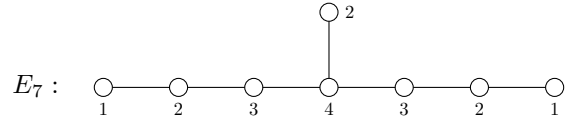
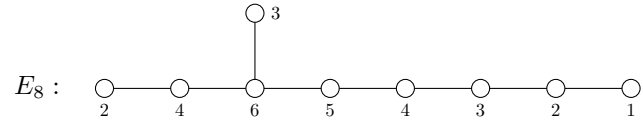
\setlength\extrarowheight{5pt}
\hskip0in
\begin{tabular}{|l|l|l|l|l|l|l|l|}
\hline
  w&   d    &  Kac diagram &good&Centralizer\\\hline
$1254$                            &  18  &  $11111$   & $\Delta^2$&*\\\hline
$123143$                  &  12  &  $11011$   &$\Delta^2$&$A_1$\\\hline
  $45423145$                   &  10  &  $01011$  &  $\Delta^2\Delta_4^8$&$2A_1$\\\hline
$1231431543165431$   &  $6$ & $ 1 1 0 0 0$   &$\Delta^2\Delta_{1356}^4$&$B_3$\\\hline
  $425423456542345$   &  6   &  $00100$    &$\Delta^2\Delta_{2345}^2$&$2A_2$\\\hline
  $23423465423456$   &  6    &    $00011$     & $\Delta^2\Delta_{24}^4$&$A_3$\\\hline
  $124315436543$   & $6$ & $01001$ &$\Delta^2$ & $A_1+B_2$ \\\hline
$142314354231365431$ &  4  &  $00010$   &$\Delta^2$&$A_1+A_3$\\\hline
  $w_0$     & $2$ &   $00001$ & $\Delta^2$&$C_4$\\\hline
\end{tabular}

\newpage

$$
E_6:\quad \begin{dynkinDiagram}[Kac,extended,root radius=.12cm, edge length=1.0cm]{E}{6}
\dynkinLabelRoots{1,1,2,3,2,1,2}
\end{dynkinDiagram}
$$

\medskip

\hskip.4in
\begin{tabular}{|l|l|l|l|l|l|}
  \hline
  Name&   d &  Kac diagram &good&Centralizer\\\hline
      &         &      $\phantom{11}1$&&\\
      &         &      $\phantom{11}1$&&\\
  $E_6$ & $12$ &   $11111$ & $\Delta^2$&$*$ \\\hline
        &          &      $\phantom{11}1$&&\\
      &         &      $\phantom{11}1$&&\\
  $E_6(a1)$ & $9$ &   $11011$ & $\Delta^2$&$A_1$ \\\hline
        &         &      $\phantom{11}1$&&\\
      &         &      $\phantom{11}0$&&\\
  $E_6(a2)$ & $6$ &  $10101$ & $\Delta^2$&$3A_1$ \\\hline
             &  &      $\phantom{11}0$&&\\
&           &          $\phantom{11}1$&&\\
  $A_5+A_1$ & $6$ &  $01010$ & $\Delta^2\Delta_{24}^4$&$4A_1$ \\\hline
            &  &      $\phantom{11}0$&&\\
 &          &          $\phantom{11}0$&&\\
  $3A_2$ & $3$ &   $00100$ & $\Delta^2$&$3A_2$ \\\hline
\end{tabular}

\newpage

$$
E_7:\quad \begin{dynkinDiagram}[Kac,extended,root radius=.12cm, edge length=1.0cm]{E}{7}
\dynkinLabelRoots{1,2,3,4,3,2,1,2}
\end{dynkinDiagram}
$$

\bigskip

{\hskip.3in
\begin{tabular}{|l|l|l|l|l|}
  \hline
  
  Name&   d &  Kac diagram &good&Centralizer\\\hline
        &          &      $\phantom{111}1$&&\\
  $E_7$     & $18$ &  $1111111$ & $\Delta^2$&$*$ \\\hline
          &         &      $\phantom{111}1$&&\\
  $E_7(a1)$ & $14$ &  $1110111$ & $\Delta^2$&$A_1$\\\hline
                &    &      $\phantom{111}1$&&\\
  $E_7(a2)$ & $12$ &  $1101011$ & $\Delta^2\Delta_{257}^2$&$2A1$ \\\hline
      &      &          $\phantom{111}1$&&\\
  $E_7(a3)$ & $30$ &   $3212123$ & $\Delta^2\Delta_{24}^4$&$*$ \\\hline
        &      &          $\phantom{111}1$&&\\
  $D_6+A_1$ & $10$ &  $0101010$ & $\Delta^2\Delta_{24}^8$&$4A_1$ \\\hline
        &      &         $\phantom{111}0$&&\\
  $A_7 $     & $8$ &  $0101010$ & $\Delta^2\Delta_{257}^2\Delta_{2}^4$&$2A_1+A_2$ \\\hline
        &      &          $\phantom{111}0$&&\\
  $E_7(a4)$ & $6$ &   $1001001$ & $\Delta^2$&$2A_2+A_1$ \\\hline
        &      &         $\phantom{111}1$&&\\
  $D_6(a2)+A_1$ & $6$ &  $0100010$ & $\Delta^2\Delta_{13}^4$&$2A_1+A_3$ \\\hline
        &      &        $\phantom{111}0$&&\\
  $A_5+A_2$  & $6$ &   $0010100$ & $\Delta^2\Delta_{2345}^2$&$3A_2$ \\\hline
        &      &         $\phantom{111}1$&&\\
  $D_4+3A_1$ & $6$ &   $0001000$ & $\Delta^2\Delta_{24567}^4$&$2A_3$ \\\hline
        &      &          $\phantom{111}0$&&\\
  $2A_3+A_1$ & $4$ &  $0001000$ & $\Delta^2\Delta_{257}^2$&$2A_3+A_1$ \\\hline        
        &      &         $\phantom{111}1$&&\\
  $7A_1$     & $2$ &  $00000000$ & $\Delta^2$&$A_7$ \\\hline
\end{tabular}
}

\newpage
$$
E_8:\quad \begin{dynkinDiagram}[Kac,extended,root radius=.12cm, edge length=1.0cm]{E}{8}
\dynkinLabelRoots{1,2,3,4,5,6,4,2,3}
\end{dynkinDiagram}
$$

\medskip

\hskip.2in
\begin{tabular}{|l|l|l|l|l|l|l|l|}
\hline
Name&   d &  Kac diagram &good&\\\hline                    
        &      &          $\phantom{11}1$&&\\
  $E_8$ & $30$ &  $11111111$ & $\Delta^2$&$*$ \\\hline
        &      &          $\phantom{11}0$&&\\
  $D_8(a2)$ & $30$   & $12102030$ &$\Delta^6\Delta_{2456}^4\Delta_5^{20}$ &$4A_1$ \\\hline
        &      &          $\phantom{11}1$&&\\
$E_8(a1)$ & $24$  &  $11011111$ & $\Delta^2$ &$A_1$ \\\hline
        &      &          $\phantom{11}1$&&\\
$E_8(a2)$ & $20$  & $11010111$ & $\Delta^2$ & $2A_1$ \\\hline
        &      &          $\phantom{11}0$&&\\
$E_7A_1$ & $18$  & $11101010$ & $\Delta^2\Delta_{24}^{16}$ & $4A_1$ \\\hline
        &      &          $\phantom{11}1$&&\\
$E_8(a4)$ & $18$  & $01010111$ & $\Delta^2\Delta_{24}^4$ & $3A_1$ \\\hline
        &      &          $\phantom{11}0$&&\\
$E_8(a5)$ & $15$  &  $10101011$ & $\Delta^2$ & $4A_1$ \\\hline
        &      &          $\phantom{11}0$&&\\
$D_8$ & $14$ & $10101010$ & $\Delta^2\Delta_{2}^{12}$ & $5A_1$ \\\hline
        &      &          $\phantom{11}0$&&\\
$E_8(a3)$ & $12$  &$10100101$ & $\Delta^2$ & $3A_1+A_2$ \\\hline
        &      &          $\phantom{11}0$&&\\
$E_8(a7)$ & $12$ & $01010011$ & $\Delta^2\Delta_{2345}^{2}$ & $A_1+2A_2$ \\\hline
        &      &          $\phantom{11}1$&&\\
$E_6+A_2$ & $12$ &  $00100100$ & $\Delta^2\Delta_{2345}^{6}$ & $3A_2$ \\\hline
        &      &          $\phantom{11}0$&&\\
$D_5(a1)+A_3$ & $12$ & $01002000$ & $\Delta^2\Delta_{123456}^{2}$ & $3A_1+A_3$ \\\hline
        &      &          $\phantom{11}0$&&\\
$D_8(a1)$ & $12$ & $00101010$ & $\Delta^2\Delta_{4578}^{4}$ & $4A_1+A_2$ \\\hline

\end{tabular}

\pagestyle{headings}
\hskip0in
\begin{tabular}{|l|l|l|l|l|l|l|l|}
\hline
  Name&   d&  Kac diagram &good&Centralizer\\\hline
          &      &          $\phantom{11}0$&&\\
$E_7(a2)+A_1$ & $12$  &$11001010$ & $\Delta^2\Delta_{2345}^{2}\Delta_{24}^{8}$ & $2A_1+A_3$ \\\hline
          &      &          $\phantom{11}0$&&\\
$E_8(a6)$ & $10$  &$00100101$ & $\Delta^2$ & $2A_1+2A_2$ \\\hline
          &      &          $\phantom{11}0$&&\\
$D_6+2A_1$ & $10$  &$11001000$ & $\Delta^2\Delta_{2456}^{8}$ & $2A_3$ \\\hline
          &      &          $\phantom{11}0$&&\\
$A_8$ & $9$  &$00100100$ & $\Delta^2\Delta_{34}^{4}$ & $A_1+3A_2$ \\\hline
        &      &          $\phantom{11}0$&&\\
$A_1+A_7$ & $8$  &$01001000$ & $\Delta^2\Delta_{2345}^{2}\Delta_{25}^{4}$ & $A_1+2A_3$ \\\hline
        &      &          $\phantom{11}0$&&\\
$D_8(a3)$ & $8$  &$00100010$ & $\Delta^2$ & $2A_1+2A_3$ \\\hline
        &      &          $\phantom{11}0$&&\\
$E_8(a8)$ & $6$  &$00010001$ & $\Delta^2$ & $A_3+A_4$ \\\hline

        &      &          $\phantom{11}0$&&\\
$E_7(a4)+A_1$ & $6$  &$01000010$ & $\Delta^2\Delta_{34}^{4}$ & $2A_1+A_5$ \\\hline
        &      &          $\phantom{11}1$&&\\
$E_6(a2)+A_2$ & $6$  &$00000100$ & $\Delta^2\Delta_{2345}^{2}$ & $A_2+A_5$ \\\hline
        &      &          $\phantom{11}0$&&\\
$2D_4$ & $6$  &$10001000$ & $\Delta^2\Delta_{1367}^{4}$ & $A_3+D_4$ \\\hline
        &      &          $\phantom{11}0$&&\\
$D_4+4A_1$ & $6$  &$11000000$ & $\Delta^2\Delta_{123456}^{4}$ & $A_7$ \\\hline
        &      &          $\phantom{11}0$&&\\
$A_1+A_2+A_5$ & $6$  &$00100000$ & $\Delta^2\Delta_{234578}^{2}\Delta_{78}^{2}$ & $A_1+A_2+A_5$ \\\hline
        &      &          $\phantom{11}0$&&\\
$2A_4$ & $5$  &$00010000$ & $\Delta^2$ & $2A_4$ \\\hline
        &      &          $\phantom{11}0$&&\\
$2A_3+2A_1$ &  $4$ &$01000000$ & $\Delta^2\Delta_{2345}^{2}$ & $A_1+A_7$ \\\hline
        &      &          $\phantom{11}0$&&\\
$2D_4(a1)$ & $4$  & $00001000$ & $\Delta^2$ & $A_3+D_5$ \\\hline
        &      &          $\phantom{11}1$&&\\
$4A_2$ & $3$  &$00000000$ & $\Delta^2$ & $A_8$ \\\hline
        &      &          $\phantom{11}0$&&\\  
  $8A_1$ & $2$ &$10000000$ & $\Delta^2$ & $D8$ \\\hline
                                          
\end{tabular}

\bibliographystyle{plain}

\begin{bibdiv}
\begin{biblist}

\bib{atlas_software}{article}{
       title={Atlas of lie groups and representations software package},
        date={2019},
        note={www.liegroups.org},
}

\bib{AH}{article}{
      author={Adams, Jeffrey},
      author={He, Xuhua},
       title={Lifting of elements of {W}eyl groups},
        date={2017},
        ISSN={0021-8693},
     journal={J. Algebra},
      volume={485},
       pages={142\ndash 165},
         url={http://dx.doi.org/10.1016/j.jalgebra.2017.04.018},
      review={\MR{3659328}},
}

\bib{bourbaki_4-6}{book}{
      author={Bourbaki, Nicolas},
       title={Lie groups and {L}ie algebras. {C}hapters 4--6},
      series={Elements of Mathematics (Berlin)},
   publisher={Springer-Verlag, Berlin},
        date={2002},
        ISBN={3-540-42650-7},
         url={https://doi.org/10.1007/978-3-540-89394-3},
        note={Translated from the 1968 French original by Andrew Pressley},
      review={\MR{1890629}},
}

\bib{bouwknegt}{article}{
      author={Bouwknegt, Peter},
       title={Lie algebra automorphisms, the {W}eyl group, and tables of shift
  vectors},
        date={1989},
        ISSN={0022-2488},
     journal={J. Math. Phys.},
      volume={30},
      number={3},
       pages={571\ndash 584},
         url={https://doi-org.proxy-um.researchport.umd.edu/10.1063/1.528422},
      review={\MR{984692}},
}

\bib{carter_conjugacy_classes}{article}{
      author={Carter, R.~W.},
       title={Conjugacy classes in the {W}eyl group},
        date={1972},
        ISSN={0010-437X},
     journal={Compositio Math.},
      volume={25},
       pages={1\ndash 59},
      review={\MR{0318337}},
}

\bib{CH}{article}{
      author={Ciubotaru, Dan},
      author={He, Xuhua},
       title={The cocenter of the graded affine {H}ecke algebra and the density
  theorem},
        date={2016},
        ISSN={0022-4049},
     journal={J. Pure Appl. Algebra},
      volume={220},
      number={1},
       pages={382\ndash 410},
  url={https://doi-org.proxy-um.researchport.umd.edu/10.1016/j.jpaa.2015.06.018},
      review={\MR{3393467}},
    }

    \bib{debacker_reeder}{article}{
      author={DeBacker, Stephen},
      author={Reeder, Mark},
       title={Depth-zero supercuspidal {$L$}-packets and their stability},
        date={2009},
        ISSN={0003-486X},
     journal={Ann. of Math. (2)},
      volume={169},
      number={3},
       pages={795\ndash 901},
  url={https://doi-org.proxy-um.researchport.umd.edu/10.4007/annals.2009.169.795},
      review={\MR{2480618}},
}

\bib{DM}{article}{
    AUTHOR = {Digne, Fran\c{c}ois}
    author =  {Michel, Jean},
     TITLE = {Quasi-semisimple elements},
   JOURNAL = {Proc. Lond. Math. Soc. (3)},
  FJOURNAL = {Proceedings of the London Mathematical Society. Third Series},
    VOLUME = {116},
      YEAR = {2018},
    NUMBER = {5},
     PAGES = {1301--1328},
      ISSN = {0024-6115},
   MRCLASS = {20G15},
  MRNUMBER = {3805058},
MRREVIEWER = {Giancarlo Lucchini Arteche},
       DOI = {10.1112/plms.12121},
       URL = {https://doi.org/10.1112/plms.12121},
       review={\MR{3805058}}, 
}

\bib{gkp}{article}{
      author={Geck, Meinolf},
      author={Kim, Sungsoon},
      author={Pfeiffer, G\"{o}tz},
       title={Minimal length elements in twisted conjugacy classes of finite
  {C}oxeter groups},
        date={2000},
        ISSN={0021-8693},
     journal={J. Algebra},
      volume={229},
      number={2},
       pages={570\ndash 600},
  url={https://doi-org.proxy-um.researchport.umd.edu/10.1006/jabr.1999.8253},
      review={\MR{1769289}},
}

\bib{geck_michel_good}{article}{
      author={Geck, Meinolf},
      author={Michel, Jean},
       title={``{G}ood'' elements of finite {C}oxeter groups and
  representations of {I}wahori-{H}ecke algebras},
        date={1997},
        ISSN={0024-6115},
     journal={Proc. London Math. Soc. (3)},
      volume={74},
      number={2},
       pages={275\ndash 305},
  url={https://doi-org.proxy-um.researchport.umd.edu/10.1112/S0024611597000105},
      review={\MR{1425324}},
}

\bib{geck_pfeiffer}{book}{
      author={Geck, Meinolf},
      author={Pfeiffer, G\"otz},
       title={Characters of finite {C}oxeter groups and {I}wahori-{H}ecke
  algebras},
      series={London Mathematical Society Monographs. New Series},
   publisher={The Clarendon Press, Oxford University Press, New York},
        date={2000},
      volume={21},
        ISBN={0-19-850250-8},
      review={\MR{1778802}},
}

\bib{rgly}{article}{
  author={Reeder, Mark},
      author={Levy, Paul},
      author={Yu, Jiu-Kang},
      author={Gross, Benedict~H.},
       title={Gradings of positive rank on simple {L}ie algebras},
        date={2012},
        ISSN={1083-4362},
     journal={Transform. Groups},
      volume={17},
      number={4},
       pages={1123\ndash 1190},
  url={https://doi-org.proxy-um.researchport.umd.edu/10.1007/s00031-012-9196-3},
      review={\MR{3000483}},
}

\bib{he_minimal_length_double_cosets}{article}{
      author={He, Xuhua},
       title={Minimal length elements in some double cosets of {C}oxeter
  groups},
        date={2007},
        ISSN={0001-8708},
     journal={Adv. Math.},
      volume={215},
      number={2},
       pages={469\ndash 503},
  url={https://doi-org.proxy-um.researchport.umd.edu/10.1016/j.aim.2007.04.005},
      review={\MR{2355597}},
}

\bib{he_nie_minimal_finite}{article}{
      author={He, Xuhua},
      author={Nie, Sian},
       title={Minimal length elements of finite {C}oxeter groups},
        date={2012},
        ISSN={0012-7094},
     journal={Duke Math. J.},
      volume={161},
      number={15},
       pages={2945\ndash 2967},
  url={https://doi-org.proxy-um.researchport.umd.edu/10.1215/00127094-1902382},
      review={\MR{2999317}},
}

\bib{kac_book}{book}{
      author={Kac, Victor~G.},
       title={Infinite-dimensional {L}ie algebras},
     edition={Third},
   publisher={Cambridge University Press, Cambridge},
        date={1990},
        ISBN={0-521-37215-1; 0-521-46693-8},
  url={https://doi-org.proxy-um.researchport.umd.edu/10.1017/CBO9780511626234},
      review={\MR{1104219}},
}

\bib{ls}{article}{
      author={Langlands, R.~P.},
      author={Shelstad, D.},
       title={On the definition of transfer factors},
        date={1987},
        ISSN={0025-5831},
     journal={Math. Ann.},
      volume={278},
      number={1-4},
       pages={219\ndash 271},
      review={\MR{MR909227 (89c:11172)}},
}
\bib{levy_exceptional}{article}{
      author={Levy, Paul},
       title={K{W}-sections for {V}inberg's {$\theta$}-groups of exceptional
  type},
        date={2013},
        ISSN={0021-8693},
     journal={J. Algebra},
      volume={389},
       pages={78\ndash 97},
  url={https://doi-org.proxy-um.researchport.umd.edu/10.1016/j.jalgebra.2013.04.035},
      review={\MR{3065993}},
    }
\bib{ov}{book}{
      author={Onishchik, A.~L.},
      author={Vinberg, {\`E}.~B.},
       title={Lie groups and algebraic groups},
      series={Springer Series in Soviet Mathematics},
   publisher={Springer-Verlag},
     address={Berlin},
        date={1990},
        ISBN={3-540-50614-4},
        note={Translated from the Russian and with a preface by D. A. Leites},
      review={\MR{91g:22001}},
}

\bib{panyushev}{article}{
      author={Panyushev, Dmitri~I.},
       title={On invariant theory of {$\theta$}-groups},
        date={2005},
        ISSN={0021-8693},
     journal={J. Algebra},
      volume={283},
      number={2},
       pages={655\ndash 670},
  url={https://doi-org.proxy-um.researchport.umd.edu/10.1016/j.jalgebra.2004.03.032},
      review={\MR{2111215}},
}

\bib{reeder_torsion}{article}{
      author={Reeder, Mark},
       title={Torsion automorphisms of simple {L}ie algebras},
        date={2010},
        ISSN={0013-8584},
     journal={Enseign. Math. (2)},
      volume={56},
      number={1-2},
       pages={3\ndash 47},
  url={https://doi-org.proxy-um.researchport.umd.edu/10.4171/LEM/56-1-1},
      review={\MR{2674853}},
    }

    \bib{serre_linear}{book}{
      author={Serre, Jean-Pierre},
       title={Linear representations of finite groups},
   publisher={Springer-Verlag, New York-Heidelberg},
        date={1977},
        ISBN={0-387-90190-6},
        note={Translated from the second French edition by Leonard L. Scott,
  Graduate Texts in Mathematics, Vol. 42},
      review={\MR{0450380}}}

\bib{springer_regular}{article}{
      author={Springer, T.~A.},
       title={Regular elements of finite reflection groups},
        date={1974},
        ISSN={0020-9910},
     journal={Invent. Math.},
      volume={25},
       pages={159\ndash 198},
  url={https://doi-org.proxy-um.researchport.umd.edu/10.1007/BF01390173},
  review={\MR{0354894}},
}





\end{biblist}
\end{bibdiv}

\end{document}